\documentclass[11pt,a4paper,leqno]{amsart}
\usepackage{amsmath,amssymb,amsthm}
\usepackage{graphicx}
\usepackage{cite,mathrsfs}

\usepackage{verbatim}
\usepackage{amssymb}
\usepackage{amsbsy}
\usepackage{amscd}
\usepackage{amsmath}
\usepackage{amsthm}
\usepackage[mathscr]{eucal}

\newtheorem{theorem}{Theorem}
\newtheorem{prop}[theorem]{Proposition}

\newtheorem{defn}{Definition}\numberwithin{defn}{section}

\theoremstyle{remark}
\newtheorem{remark}[theorem]{Remark}
\newtheorem{example}[theorem]{Example}

\numberwithin{equation}{section}
\allowdisplaybreaks

\def\abs#1{\lvert#1\rvert}
\long\def\comment#1{}

\def\VN#1{\overset{\rlap{$#1$}}{\V}}
\def\EB#1#2{\rlap{\raisebox{-0.3ex}{#2}}\raisebox{0.3ex}{#1}}
\def\EDV#1#2{\rlap{\rotatebox{10}{#1\V}}\rotatebox{-10}{#2\V}}
\def\EDVN#1#2#3#4{\rlap{\rotatebox{10}{#1\rotatebox{-10}{$\overset{\rlap{$#3$}}{\V}$}}}\rotatebox{-10}{#2\rotatebox{10}{$\underset{\rlap{$#4$}}{\V}$}}}
\def\V{{\setlength{\unitlength}{1pt}\begin{picture}(4,4)\put(2,2){\circle{4}}\end{picture}}}
\def\E{{\setlength{\unitlength}{1pt}\begin{picture}(38,2)\put(0,2){\line(1,0){38}}\end{picture}}}
\def\EO{{\setlength{\unitlength}{1pt}\begin{picture}(38,6)\put(0,2){\line(1,0){16}}\put(22,2){\line(1,0){16}}\put(16,4){\oval(4,4)[l]}\put(16,0){\oval(4,4)[r]}\put(22,4){\oval(4,4)[l]}\put(22,0){\oval(4,4)[r]}\end{picture}}}

\title[Functional relations for zeta-functions of root systems]{An overview and supplements to the theory of functional relations for zeta-functions of root systems}

\author[Y. Komori]{Yasushi Komori}
\author[K. Matsumoto]{Kohji Matsumoto}
\author[H. Tsumura]{Hirofumi Tsumura}

\address{Y. Komori: Department of Mathematics, Rikkyo University, Nishi-Ikebukuro, Toshima-ku, Tokyo 171-8501, Japan}
\email{komori@rikkyo.ac.jp}

\address{K. Matsumoto: Graduate School of Mathematics, Nagoya University, Chikusa-ku, Nagoya 464-8602, Japan}
\email{kohjimat@math.nagoya-u.ac.jp}

\address{{H.\,Tsumura:} Department of Mathematical Sciences, Tokyo Metropolitan University, 1-1, Minami-Ohsawa, Hachioji, Tokyo 192-0397, Japan}
\email{tsumura@tmu.ac.jp}



\begin{document}

\begin{abstract}
We give an overview of the theory of functional relations for zeta-functions of root
systems, and show some new results on functional relations involving zeta-functions
of root systems of types $B_r$, $D_r$, $A_3$ and $C_2$.
To show those new results, we use two different methods.    
The first method, for $B_r$, $D_r$, $A_3$, is via generating functions, 
which is based on the symmetry with respect to Weyl groups, or more generally, on our
theory of lattice sums of certain hyperplane arrangements.
The second method for $C_2$ is more elementary, using partial fraction decompositions.
\end{abstract}

\maketitle


\section{Introduction}\label{sec-1}

Let $\mathbb{N}$ be the set of positive integers, $\mathbb{N}_0$ the set of
non-negative integers, $\mathbb{Z}$ the set of rational integers, $\mathbb{R}$ the set of real numbers, and $\mathbb{C}$ the set of complex numbers. 

Let $V$ be an $r$-dimensional real vector space with the inner product
$\langle\;,\;\rangle$, and $\Delta\subset V$ be a reduced root system of rank $r$.
Let $\Psi=\{\alpha_1,\ldots,\alpha_r\}$ be its fundamental system.
Denote by $\Delta_+, \Delta_-$ the set of
all positive roots and of all negative roots, respectively, and $n=|\Delta_+|$.    
For any $\alpha\in\Delta$, we denote by $\alpha^{\vee}$ the associated coroot.

Let $\Lambda=\{\lambda_1,\ldots,\lambda_r\}$ be the set of fundamental weights
defined by $\langle\alpha_i^{\vee},\lambda_j\rangle=\delta_{ij}$ (Kronecker's delta).
Let $Q^{\vee}$ be the coroot lattice,
$P$ the weight lattice, $P_+$ the set of integral dominant weights, and
$P_{++}$ the set of integral strongly dominant weights, respectively, defined 
by
\begin{align*}
&Q^{\vee}=\bigoplus_{i=1}^r \mathbb{Z}\alpha_i^{\vee},\quad
P=\bigoplus_{i=1}^r\mathbb{Z}\lambda_i,\\
&P_+=\bigoplus_{i=1}^r\mathbb{N}_0\lambda_i,\quad
P_{++}=\bigoplus_{i=1}^r\mathbb{N}\lambda_i.
\end{align*}

The zeta-function of the root system $\Delta$
is defined by an $r$-ple series in $n$ variables.
It is defined by
\begin{align}\label{def_zeta}
\zeta_r({\bf s};\Delta)=\sum_{m_1=1}^{\infty}\cdots\sum_{m_r=1}^{\infty}
\prod_{\alpha\in\Delta_+}\langle\alpha^{\vee},m_1\lambda_1+\cdots+m_r\lambda_r\rangle
^{-s_{\alpha}},
\end{align}
where ${\bf s}=(s_{\alpha})_{\alpha\in\Delta_+}$ is a complex vector.
This multiple series is convergent absolutely when $\Re s_{\alpha}>1$ for all $\alpha$, 
and can be continued meromorphically to the whole space $\mathbb{C}^n$ 
(see \cite[Theorem 3]{MatBonn}).
When the root system is of type $X_r$ ($X=A,B,C,D,E,F$ or $G$), 
we write the associated zeta-function as $\zeta_r({\bf s};X_r)$.

When $X=A,B,C$ and $D$, the explicit form of \eqref{def_zeta} was given in
\cite{KMTWitten2}, where the recursive structure of those zeta-functions was also
discussed.    The case of type $G_2$ was studied in \cite{KMTWitten4,KMTWitten5}.

The notion of zeta-functions of root systems is a generalization of the following three
kinds of multiple zeta-functions.

$1^{\circ}$.
It is a multi-variable generalization of Witten zeta-functions.
Let $\mathfrak{g}$ be a semisimple Lie algebra, and define
\begin{align}\label{def_Witten}
\zeta_W(s,\mathfrak{g})=\sum_{\varphi}(\dim\varphi)^{-s},
\end{align}
where the sum runs over all equivalent classes of finite dimensional irreducible
representations of $\mathfrak{g}$.  
Zagier \cite{Zagier} introduced \eqref{def_Witten} under the name of the Witten
zeta-function, after the work of Witten \cite{Witten}.
Since there is a one-to-one correspondence between irreducible representations and
dominant weights, using Weyl's dimension formula we find that
\begin{align}\label{Witten-root}
\zeta_W(s,\mathfrak{g})&=K(\mathfrak{g})^s 
\sum_{m_1=1}^{\infty}\cdots\sum_{m_r=1}^{\infty}
\prod_{\alpha\in\Delta_+}\langle\alpha^{\vee},m_1\lambda_1+\cdots+m_r\lambda_r\rangle
^{-s}\\
&=K(\mathfrak{g})^s\zeta_r(s,\ldots,s;\Delta(\mathfrak{g})),\notag
\end{align}
where $K(\mathfrak{g})=\prod_{\alpha\in\Delta_+}\langle\alpha^{\vee},\lambda_1+\cdots+\lambda_r\rangle$ and
$\Delta(\mathfrak{g})$ is the root system corresponding to $\mathfrak{g}$.

\begin{remark}
Witten originally considered the zeta values associated not with Lie algebras, but with
Lie groups.     Multi-variable version of such zeta-functions associated with Lie groups
has been studied in \cite{KMTPalanga,KMTKyiv}.
\end{remark}

$2^{\circ}$.
Tornheim \cite{Tornheim} considered special values of the double series
\begin{align}\label{def_Tornheim}
\zeta_{MT,2}(s_1,s_2,s_3)=\sum_{m_1=1}^{\infty}\sum_{m_2=1}^{\infty}
m_1^{-s_1}m_2^{-s_2}(m_1+m_2)^{-s_3}
\end{align}
at positive integer points.    From our viewpoint, this function is nothing but the
zeta-function of the root system of type $A_2$.    Later, the second author
\cite{MatBonn} introduced the $C_2$-analogue, that is
\begin{align*}
&\zeta_2(s_1,s_2,s_3,s_4;C_2)\\
&\quad =\sum_{m_1=1}^{\infty}\sum_{m_2=1}^{\infty}
m_1^{-s_1}m_2^{-s_2}(m_1+m_2)^{-s_3}(m_1+2m_2)^{-s_4},\notag
\end{align*}
and then the second and the third authors \cite{MTFourier} considered more general
zeta-functions of root systems of type $A_r$, which are of the form
\begin{align}\label{def_A_r}
\zeta_r({\bf s};A_r)=\sum_{m_1=1}^{\infty}\cdots\sum_{m_r=1}^{\infty}
\prod_{1\leq i<j\leq r+1}(m_i+\cdots+m_{j-1})^{-s_{ij}},
\end{align}
where ${\bf s}=(s_{ij})_{1\leq i<j\leq r+1}$.
On the other hand, as another generalization of \eqref{def_Tornheim}, the second author
\cite{MatBonn} also introduced the Mordell--Tornheim multiple zeta-function
\begin{align}\label{def_MT}
& \zeta_{MT,r}(s_1,\ldots,s_r,s_{r+1})\\
& \quad =\sum_{m_1=1}^{\infty}\cdots\sum_{m_r=1}^{\infty} m_1^{-s_1}\cdots m_r^{-s_r}(m_1+\cdots+m_r)^{-s_{r+1}}\notag
\end{align}
(Mordell \cite{Mordell} considered the special case $s_1=\cdots=s_r=s_{r+1}=1$ in
\eqref{def_MT}, with an additional constant term on the last factor).
This class of multiple zeta-functions has also been studied by several 
mathematicians recently.
We can see that \eqref{def_A_r} is a generalization of \eqref{def_MT}, because if we
put $s_{ij}=0$ for all $(i,j)\neq (1,2),(2,3),\ldots,(r,r+1), (1,r+1)$ in \eqref{def_A_r},
then it reduces to \eqref{def_MT}.

$3^{\circ}$.
The Euler--Zagier multiple zeta-function is defined by
\begin{align}\label{def_EZ}
& \zeta_{EZ,r}(s_1,\ldots,s_r)\\
& \quad =\sum_{m_1=1}^{\infty}\cdots\sum_{m_r=1}^{\infty}
m_1^{-s_1}(m_1+m_2)^{-s_2}\cdots(m_1+\cdots+m_r)^{-s_r}\notag
\end{align}
(Hoffman \cite{Hoffman}, Zagier \cite{Zagier}).
Special values of \eqref{def_EZ} at positive integer points (in the domain of absolute
convergence) are called multiple zeta values (MZV), and have been studied extensively
in these decades.     We find that, if we put 
$s_{ij}=0$ for all $(i,j)\neq (1,2),(1,3),\ldots,(1,r+1)$ in \eqref{def_A_r},
then it reduces to \eqref{def_EZ}.    It is also possible to understand \eqref{def_EZ}
as special cases of zeta-functions of root systems of type $C_r$ (see \cite{KMTFACM}).

From the above observation we can expect that we will be able to construct a unified
theory of multiple zeta-functions from the viewpoint of root systems.
This expectation is embodied when we consider the problem of functional relations, which
we now explain.

In the study of MZV, a central problem is to find various relations among those values.
In fact, a lot of such relations are known (duality, sum formula, Ohno relation,
Le--Murakami relation, \dots).
Then it is a natural question to ask: whether those relations are valid only at integer 
points, or valid also at other values continuously as functional relations?
This question was raised by the second author (cf. \cite[p.161]{Mat06}).
In the frame of Euler--Zagier multiple zeta-functions, no such functional relation
is known, except for the classical harmonic product formula
(and perhaps the functional equations for double zeta-functions, see \cite{Mat04},
\cite{KMTDebrecen,KMTIJNT}).
However, if we observe the landscape from the wider standpoint of zeta-functions of
root systems, we are able to find various functional relations whose specialization
gives relations among MZVs.
The aim of the present article is to survey the known results on functional relations
among zeta-functions of root systems, and report some new results in this direction.

In the next section, we summarize various previous results on this topic.
In particular we mention the main general result
(Theorem \ref{thm:main1}) in \cite{KMT-Poincare}.
Then in Sections 3 to 5, we give some explicit examples of this Theorem \ref{thm:main1}
in the cases of types $B_r$, $D_r$, and $A_3$.
In the last section we present another method of obtaining functional relations,
for zeta-functions of type $C_2$.

\section{A survey on previous methods} \label{sec-2}

The first affirmative answer to the question mentioned at the end of the introduction 
is the following result of the third author \cite{Tsu07}. Let $\zeta(s)$ be the Riemann zeta-function and let
$$\zeta_{2}(s_1,s_2,s_3;A_2)=\zeta_{MT,2}(s_1,s_2,s_3)$$
defined by \eqref{def_Tornheim}. Then, for $k,l\in\mathbb{N}_0$, 
\begin{align}\label{Tsumura_formula}
\zeta_2(k,l,s;A_2)+(-1)^k\zeta_2(k,s,l;A_2)+(-1)^l\zeta_2(l,s,k;A_2)\\
=2\sum_{j=0}^{[k/2]}\binom{k+l-2j-1}{l-1}\zeta(2j)\zeta(s+k+l-2j)\notag\\
+2\sum_{j=0}^{[l/2]}\binom{k+l-2j-1}{k-1}\zeta(2j)\zeta(s+k+l-2j)\notag
\end{align}
holds for any $s\in\mathbb{C}$.

The above formula is actually a little different from the form given in \cite{Tsu07}.
Inspired by \cite{Tsu07}, Nakamura \cite{Nak06} published an alternative method, which 
gives the above form.     From the expression in \cite{Tsu07}, it is possible to deduce
the above form, just by using a certain elementary lemma (see \cite[Lemma 2.1]{MNOT}).

The basic idea in \cite{Tsu07} is to consider the series with additional factor
$(-u)^{-n}$, where $u\geq 1$ and $n\in\mathbb{N}$.    Because of this additional factor, 
the series is convergent nicely.   And at the end of the proof take the limit $u\to 1$
to obtain the relation of ordinary multiple zeta-functions.    This method is sometimes
called the $u$-method.
Using the same method, the second and the third authors \cite{MTFourier} proved
some functional relations involving the zeta-function of $A_3$.

Nakamura's method in \cite{Nak06} is different.    His argument starts with the expression
\begin{align}\label{Nak_fund}
\sum_{\substack{m_1\neq 0, m_2\neq 0 \\ m_1+m_2\neq 0}}
\lefteqn{m_1^{-s_1}m_2^{-s_2}(m_1+m_2)^{-s_3}}\\
&=\int_0^1 \sum_{m_1\neq 0}\frac{e^{2\pi im_1 x}}{m_1^{s_1}}
\sum_{m_2\neq 0}\frac{e^{2\pi im_2 x}}{m_2^{s_2}}
\sum_{n\neq 0}\frac{e^{-2\pi in x}}{n^{s_3}}dx,\notag
\end{align}
which was used by Zagier in his lecture at Kyushu University in 1999, and then uses
properties of Bernoulli polynomials.
By the same technique, Nakamura \cite{Nak08} obtained functional relations among
zeta-functions of $B_2$, $A_3$, and of $A_2$ with characters.

Yet another (rather elementary) method was proposed by Zhou, Bradley and Cai \cite{ZBC}, 
who stated certain
functional relations among zeta-functions of $A_3$.     Inspired by \cite{ZBC}, Ikeda
and Matsuoka \cite{IM} gave functional relations among zeta-functions of
$A_2, A_3$ and $A_4$ (with correcting some inaccurate point in \cite{ZBC}).

All of the methods mentioned above have the same feature, that is, in general, some of
the variables in the results are forced to be 0. 
To avoid this unsatisfactory restriction, the authors \cite{KMTKyushu} introduced the idea
of considering infinite series of polylogarithm type (that is, with an additional 
exponential factor on the numerator), and combining this additional flexibility
with the $u$-method to obtain
more general form of functional relations.   This technique was inspired by the work of
the second and the third authors \cite{MTSiauliai}.
Applying this idea, the authors proved various functional relations among zeta-functions of
$A_3, B_3, C_2, C_3$ and $G_2$ (see \cite{KMTNicchuu,KMTWitten4,KMTWitten5}).

It is noted that the second and the third authors \cite{MTQJM} introduced another new method of finding functional relations among certain multiple zeta-functions including the zeta-function of type $A_2$. This method can be regarded as a kind of multiple analogue of Hardy's one (see \cite{Hardy}, also \cite[Section 2.2]{Titch}) of proving the functional equation for the Riemann zeta-function. However, at present, it is unclear whether we can apply this method to zeta-functions of root systems of general types.

Another approach to functional relations for the zeta-function of type $A_2$, due to
Onodera \cite{Onod14} is also to be mentioned.

On the other hand, when one observes the form of the left-hand side of
\eqref{Tsumura_formula}, one may feel that there may be some underlying connection with
the action of the associated Weyl group.    This is in fact true, and the authors
developed an alternative (more structural) approach of finding functional relations.

Let $I\subset\{1,2,\ldots,r\}$, and $\Psi_I=\{\alpha_i \;|\; i\in I\}\subset \Psi$.
Let $V_I$ be the subspace of $V$ spanned by $\Psi_I$.    Then $\Delta_I=\Delta\cap V_I$
is the root system in $V_I$ whose fundamental system is $\Psi_I$.
For $\Delta_I$, we denote the corresponding coroot lattice, weight lattice etc.\ by 
$Q_I^{\vee}=\bigoplus_{i\in I}\mathbb{Z}\alpha_i^{\vee}$,
$P_I=\bigoplus_{i\in I}\mathbb{Z}\lambda_i$ etc. 
Let $\iota:Q_I^{\vee}\to Q^{\vee}$ be the natural embedding, 
and $\iota^*:P\to P_I$ the projection induced from $\iota$; that is, for 
$\lambda\in P$, $\iota^*(\lambda)$ is defined as a unique element of $P_I$ satisfying
$\langle\iota(q),\lambda\rangle=\langle q,\iota^*(\lambda)\rangle$ for all
$q\in Q_I^{\vee}$.

Let $\sigma_{\alpha}$ be the reflection with respect to $\alpha$, and 
denote by $W=W(\Delta)$ the Weyl
group of $\Delta$, namely the group generated by $\{\sigma_{i}\;|\;1\leq i\leq r\}$,
where $\sigma_i=\sigma_{\alpha_i}$.
For $w\in W$, we put $\Delta_w=\Delta_+\cap w^{-1}\Delta_-$.
Let $W_I$ be the subgroup of $W$ generated by all the reflections associated with the elements
in $\Psi_I$, and
$W^I=\{w\in W\;|\; \Delta_{I+}^{\vee}\subset w\Delta_+^{\vee}\}$.

The fundamental Weyl chamber is defined by
$$
C=\{v\in V\;|\; \langle\alpha_i^{\vee},v\rangle\geq 0\;{\rm for}\;1\leq i\leq r\}.
$$
Then $W$ acts on the set of Weyl chambers $\{wC\;|\;w\in W\}$ simply transitively.
For any subset $A\subset\Delta$, 
let $H_{A^{\vee}}$ be the set of all
$v\in V$ which satisfies $\langle\alpha^{\vee},v\rangle= 0$ for some $\alpha\in A$.  
In particular, $H_{\Delta^{\vee}}$ is the set of all walls of Weyl chambers.

For $\mathbf{s}=(s_{\alpha})_{\alpha\in \Delta_+}\in\mathbb{C}^n$, 
define the action of $W$ to $\mathbf{s}$ 
by $(w\mathbf{s})_{\alpha}=s_{w^{-1}\alpha}$ for $w\in W$ with the convention 
that, if $\alpha\in\Delta_-$, then we understand that $s_{\alpha}=s_{-\alpha}$.
Define
\begin{align}\label{2-1}
S(\mathbf{s},\mathbf{y};I;\Delta)=\sum_{\lambda\in \iota^{*-1}(P_{I+})\setminus 
  H_{\Delta^{\vee}}}
  e^{2\pi\sqrt{-1}\langle \mathbf{y},\lambda\rangle}
  \prod_{\alpha\in\Delta_+}
  \frac{1}{\langle\alpha^\vee,\lambda\rangle^{s_\alpha}},
  \end{align}
where $\mathbf{y}\in V$.
This sum was first introduced in \cite[(110)]{KMTWitten3}.
Moreover in \cite[Theorems 5 and 6]{KMTWitten3}, we showed
\begin{align}\label{sahen}
S(\mathbf{s},\mathbf{y};I;\Delta)
    =
    \sum_{w\in W^I}
    \Bigl(\prod_{\alpha\in\Delta_{w^{-1}}}(-1)^{-s_{\alpha}}\Bigr)
    \zeta_r(w^{-1}\mathbf{s},w^{-1}\mathbf{y};\Delta).
\end{align}

Also in the same article we proved a certain multiple integral expression of
$S(\mathbf{s},\mathbf{y};I;\Delta)$, and noticed that when $I=\emptyset$ and $s_{\alpha}$ 
are positive integers, 
the integrand of that expression can be written in terms of classical Bernoulli 
polynomials.    
This observation led us to the definition of Bernoulli functions of root systems and their
generating functions.
However it requires extremely huge task if we want to
calculate the integral expression given in \cite{KMTWitten3} more explicitly.
This situation was improved in \cite{KMTLondon}, in which more accessible expressions were given when $I=\emptyset$ (see \cite[Theorem 4.1]{KMTLondon}).

To define Bernoulli functions of root systems, we need some more notations.

Let 
$\Delta^*=\Delta_+\setminus\Delta_{I+}$ and
$d=\abs{I^c}$.
We may find $\mathbf{V}_I=\{\gamma_1,\ldots,\gamma_d\}\subset \Delta^*$
such that $\mathbf{V}=\mathbf{V}_I\cup\Psi_I$ becomes a basis of $V$.
Let $\mathscr{V}_I=\mathscr{V}(\Delta^*)$ be the set of all such bases.
In particular, $\mathscr{V}=\mathscr{V}_{\emptyset}$ be the set of all
linearly independent subsets
$\mathbf{V}=\{\beta_1,\ldots,\beta_r\}\subset\Delta_+$.

For $\mathbf{V}\in\mathscr{V}_I$,
the lattice $L(\mathbf{V}^{\vee})=\bigoplus_{\beta\in\mathbf{V}}\mathbb{Z}\beta^{\vee}$
is a sublattice of $Q^{\vee}$.
Let $\{\mu^{\mathbf{V}}_\gamma\}_{\gamma\in\mathbf{V}}$ 
be the dual basis of $\mathbf{V}^\vee=\mathbf{V}_I^\vee\cup\Psi_I^\vee$,
namely
$\langle \gamma_k^{\vee},\mu^{\mathbf{V}}_{\gamma_l}\rangle=\delta_{kl}$,
$\langle \alpha_i^{\vee},\mu^{\mathbf{V}}_{\alpha_j}\rangle=\delta_{ij}$, and
$\langle\gamma_k^{\vee}, \mu^{\mathbf{V}}_{\alpha_i}\rangle
=\langle\alpha_i^{\vee}, \mu^{\mathbf{V}}_{\gamma_k}\rangle=0$.
Let $p_{\mathbf{V}_I^\perp}$ 
be the projection defined by
\begin{equation}
  \label{eq:proj}
  p_{\mathbf{V}_I^\perp}(v)=
  v-\sum_{\gamma\in\mathbf{V}_I}\mu^{\mathbf{V}}_\gamma\langle\gamma^\vee,v\rangle
  =
  \sum_{\alpha\in\Psi_I}\mu^{\mathbf{V}}_{\alpha}\langle\alpha^\vee,v\rangle,
\end{equation}
for $v\in V$.  

Next we introduce a generalization of the notion of ``fractional part'' of real numbers.
Let $\mathscr{R}$ be the set of all linearly independent subsets
$\mathbf{R}=\{\beta_1,\ldots,\beta_{r-1}\}\subset\Delta$, and let
$\mathfrak{H}_{\mathbf{R}^{\vee}}=\bigoplus_{i=1}^{r-1}\mathbb{R}\beta_i^{\vee}$
be the hyperplane passing through $\mathbf{R}^{\vee}\cup\{0\}$.    We fix a non-zero vector
$$
\phi\in V\setminus\bigcup_{\mathbf{R}\in\mathscr{R}}\mathfrak{H}_{\mathbf{R}^{\vee}}.
$$
Then $\langle\phi,\mu_{\beta}^{\mathbf{V}}\rangle\neq 0$ for all 
$\mathbf{V}\in\mathscr{V}$ and $\beta\in\mathbf{V}$.
For $\mathbf{y}\in V$, $\mathbf{V}\in\mathscr{V}$ and $\beta\in\mathbf{V}$, we define
\begin{align}
\{\mathbf{y}\}_{\mathbf{V},\beta}=\left\{
   \begin{array}{ll}
   \{\langle\mathbf{y},\mu_{\beta}^{\mathbf{V}}\rangle\},& 
        (\langle\phi,\mu_{\beta}^{\mathbf{V}}\rangle>0),\\
   1-\{-\langle\mathbf{y},\mu_{\beta}^{\mathbf{V}}\rangle\},
        & (\langle\phi,\mu_{\beta}^{\mathbf{V}}\rangle<0),
   \end{array}
   \right.
\end{align}
where $\{\cdot\}$ on the right-hand sides denotes the usual fractional part of real
numbers.

Using these notions, we now define Bernoulli functions of the root system $\Delta$
associated with $I$ and their generating functions.

\begin{defn}[{\cite[Definition 2.2]{KMT-Poincare}}]\  
For $\mathbf{t}_I=(t_{\alpha})_{\alpha\in\Delta^*}\in\mathbb{C}^n$ and
$\lambda\in P_I$, let
\begin{align}
  \label{eq:exp_F}
 & F(\mathbf{t}_I,\mathbf{y},\lambda;I;\Delta)\\
  &=
  \sum_{\mathbf{V}\in\mathscr{V}_I}
  \left(\prod_{\gamma\in \Delta^*\setminus\mathbf{V}_I}
  \frac{t_\gamma}
  {t_\gamma-\sum_{\beta\in\mathbf{V}_I}
    t_\beta\langle\gamma^\vee,\mu^{\mathbf{V}}_\beta\rangle
    -2\pi\sqrt{-1}
    \langle \gamma^\vee,p_{\mathbf{V}_I^\perp}(\lambda)\rangle}\right)\notag
  \\
  &\qquad\times
  \frac{1}{\abs{Q^\vee/L(\mathbf{V}^\vee)}}
  \sum_{q\in Q^\vee/L(\mathbf{V}^\vee)}
  \exp(2\pi\sqrt{-1}\langle \mathbf{y}+q,p_{\mathbf{V}_I^\perp}(\lambda)\rangle)\notag\\
  &\qquad\times
  \prod_{\beta\in\mathbf{V}_I}
  \frac{t_\beta\exp
    (t_\beta
    \{\mathbf{y}+q\}_{\mathbf{V},\beta})}{e^{t_\beta}-1},\notag
\end{align}
and define {\it Bernoulli functions} $P(\mathbf{k},\mathbf{y},\lambda;I;\Delta)$ 
{\it of the root system $\Delta$
associated with $I$} by the expansion
\begin{equation}
\label{eq:def_F}
  F(\mathbf{t}_I,\mathbf{y},\lambda;I;\Delta)=
  \sum_{\mathbf{k}\in \mathbb{N}_0^{\abs{\Delta^*}}}P(\mathbf{k},\mathbf{y},\lambda;I;\Delta)
  \prod_{\alpha\in \Delta^*}
  \frac{t_\alpha^{k_\alpha}}{k_\alpha!}.
\end{equation}
\end{defn}

\begin{theorem}[{\cite[Theorem 2.3]{KMT-Poincare}}]
\label{thm:main1}
Let $s_\alpha=k_\alpha\in\mathbb{N}$ for $\alpha\in \Delta^*$ and
$s_\alpha\in\mathbb{C}$ for $\alpha\in\Delta_{I+}$. 
We assume


{\rm (\#)} If $\alpha$ belongs to an irreducible component of type $A_1$,
 then the corresponding $k_{\alpha}\geq 2$.  

Then we have
\begin{equation}
  \label{eq:func_eq}
  \begin{split}
    &S(\mathbf{s},\mathbf{y};I;\Delta)
    =(-1)^{\abs{\Delta^*}}
    \biggl(\prod_{\alpha\in \Delta^*}
    \frac{(2\pi\sqrt{-1})^{k_\alpha}}{k_\alpha!}\biggr)
    \\
    &\qquad\times\sum_{\lambda\in P_{I++}}
    \biggr(
    \prod_{\alpha\in\Delta_{I+}}
    \frac{1}{\langle\alpha^\vee,\lambda\rangle^{s_\alpha}}
    \biggl)
    P(\mathbf{k},\mathbf{y},\lambda;I;\Delta).
  \end{split}
\end{equation}
\end{theorem}

When $I=\emptyset$, this result was obtained in \cite[(3.10)]{KMTLondon}.   In this case
$P_I=P_{\emptyset}=\{\mathbf{0}\}$, and so there is only one $\lambda$, that is
$\lambda=\mathbf{0}$.
For general $I$, however, we have to consider the above sum with respect to
$\lambda\in P_{++}$, so the situation becomes much more complicated.    In fact, 
to prove Theorem \ref{thm:main1} for general $I$, we had to develop the theory of
certain lattice sums of hyperplane arrangements \cite{KMT2014}, and using the result
in \cite{KMT2014} we proved the general case of Theorem \ref{thm:main1} in
\cite{KMT-Poincare}.

Combining this theorem with \eqref{sahen}, we obtain another tool of showing functional
relations among zeta-functions of root systems.
For this purpose, it is necessary to know the explicit form of
$P(\mathbf{k},\mathbf{y},\lambda;I;\Delta)$.     In view of \eqref{eq:def_F}, this 
can be done if we know the explicit form of
$F(\mathbf{t}_I,\mathbf{y},\lambda;I;\Delta)$.
This point will be discussed in the next two sections.

\ 

\section{Generating Functions ($B_{r-1}\subset B_r$ and $D_{r-1}\subset D_r$ Cases)}\label{sec-3}

In \cite[Sections 3 and 4]{KMT-Poincare},
we gave the explicit forms of the generating function 
$F(\mathbf{t}_I,\mathbf{y},\lambda;I;\Delta)$
in the cases of the root systems of $A_r$ and $C_r$ with $I=\{2,\ldots,r\}$,
and in the cases of arbitrary root systems with $|I|=1$.

In this section we give
the generating functions in the rest cases of classical root systems, that is, $B_r$ and $D_r$ cases with $I=\{2,\ldots,r\}$, following the method in
\cite[Sections 3 and 4]{KMT-Poincare}.

The other cases, in general, seem much more complicated.
As a first step to consider general cases,
in the next section, we present the generating function of type $A_3$ with $I=\{1,3\}$. 


Let
$I\subset \{1,\ldots,r\}$ with $|I^c|=1$ and put $I^c=\{k\}$.
Then $\Psi_I=\{\alpha_i\}_{i\in I}$,
and we see that
$$
\Delta^{*\vee}=\{\alpha^\vee=\sum_{i=1}^r a_i\alpha_i^\vee\in\Delta_+^\vee~|~a_k=\langle\alpha^\vee,\lambda_k\rangle\neq 0\}.
$$ 
Since $|\mathbf{V}_I|=1$ in the present case, we have
\begin{equation}
  \mathscr{V}_I=\{\mathbf{V}=\{\beta\}\cup\Psi_I\}_{\beta\in\Delta^*}.
\end{equation}
For $\mathbf{V}=\{\beta\}\cup\Psi_I\in\mathscr{V}_I$ and
$\gamma\in\Delta^*\setminus\{\beta\}$, 
we have
the transposes $p^*_{\mathbf{V}_I^\perp}$ of the projections 
$p_{\mathbf{V}_I^\perp}$ (defined by \eqref{eq:proj}) as 
\begin{equation}\label{def-p-*}
 p^*_{\mathbf{V}_I^\perp}(\gamma^\vee)=\gamma^\vee-\langle\gamma^\vee,\mu^{\mathbf{V}}_\beta\rangle\beta^\vee.
\end{equation}
We put $b_i=b_i(\beta)=\langle\beta^\vee,\lambda_i\rangle$ ($1\leq i\leq r$) so that
\begin{equation}
  \beta^\vee=\sum_{i=1}^r b_i\alpha_i^\vee,\qquad
  \mu_\beta^{\mathbf{V}}=\frac{\lambda_k}{b_k}.  
\end{equation}
Write $\mathbf{y}=y_1\alpha_1^\vee+\cdots+y_r\alpha_r^\vee$
and $\lambda=\sum_{\substack{i=1\\i\neq k}}^r m_i\lambda_i\in P_I$. 
Then we obtain the following form of the generating function
(under the identification $\mathbf{y}=(y_i)_{1\leq i\leq r}$ and
$\lambda=(m_i)_{1\leq i(\neq k)\leq r}$):
\begin{align}\label{Fq-general}
  &  F((t_\beta)_{\beta\in\Delta^*},(y_i)_{1\leq i\leq r},(m_i)_{1\leq i(\neq k)\leq r};I;\Delta)  
\\
  &=
  \sum_{\beta\in\Delta^*}
  \prod_{\gamma\in\Delta^*\setminus\{\beta\}}\frac{t_\gamma}{t_\gamma-\frac{\langle\gamma^\vee,\lambda_k\rangle}{b_k}t_\beta-2\pi\sqrt{-1}\langle p^*_{\mathbf{V}_I^\perp}(\gamma^\vee),\lambda\rangle}\notag
  \\
  &\qquad\times
\frac{1}{b_k}\sum_{0\leq a_k<b_k}
  \exp\Bigl(2\pi\sqrt{-1}
  \sum_{\substack{i=1\\i\neq k}}^r m_i \Bigl(y_i-\frac{b_i}{b_k}(y_k+a_k)
  \Bigr)
  \Bigr)\notag\\
& \qquad \times 
  \frac{t_\beta\exp\Bigl(t_\beta\Bigl\{\dfrac{y_k+a_k}{b_k}\Bigr\}\Bigr)}{e^{t_\beta}-1}
  \notag
\end{align}
(see \cite[(25)]{KMT-Poincare}).
Using the above expressions, we will give
explicit forms of generating functions of types $B_r$ and $D_r$.

\subsection{$B_r$ Case}
We realize $\Delta_+^\vee=\{e_i\pm e_j~|~1\leq i<j\leq r\}\cup\{2e_j~|~1\leq j\leq r\}$
and $\Psi^{\vee}=\{e_1-e_2,\ldots,e_{r-1}-e_r, 2e_r\}$. 
Then 
\begin{align}\label{lambda_B}
\lambda_i=\left\{
  \begin{array}{ll}
   e_1+\cdots+e_i  &  (1\leq i\leq r-1), \\
   (e_1+\cdots+e_r)/2  &  (i=r).
  \end{array}
  \right.
  \end{align}
 
Choose $I=\{2,\ldots,r\}$ and $I^c=\{1\}$ as in the following diagram.
\begin{equation}
\setlength{\unitlength}{1pt}\begin{picture}(0,0)
\put(149,2){\oval(238,28)}
\end{picture}
\VN{\alpha_1}\E\VN{\alpha_2}\E\V\EO\V\E\V\E\V\EB{\E}{\E}\VN{\!\!\!\!\!\!\alpha_{r}}
\end{equation}
Then 
\begin{equation}
  \Psi_I^\vee=\{\alpha_2^\vee=e_2-e_3,\ldots,\alpha_{r-1}^\vee=e_{r-1}-e_r,\alpha_r^\vee=2e_r\}
\end{equation}
and $\Delta^{*\vee}=\{e_1\pm e_j~|~2\leq j\leq r\}\cup\{2e_1\}$. 
Hence 
\begin{equation}
  \mathscr{V}_I^\vee=\{\{e_1-e_j\}\cup\Psi_I^\vee\}_{2\leq j\leq r}
\cup\{\{e_1+e_j\}\cup\Psi_I^\vee\}_{2\leq j\leq r}
\cup\{\{2e_1\}\cup\Psi_I^\vee\}.
\end{equation}
For $\beta^\vee=e_1-e_j$ ($2\leq j\leq r$), we see that
\begin{equation}
  b_l=\langle\beta^\vee,\lambda_l\rangle=
  \begin{cases}
    1\qquad&(1\leq l<j)\\
    0\qquad&(j\leq l\leq r),
  \end{cases}
\end{equation}
and in particular $\mu_\beta^{\mathbf{V}}=\lambda_1=e_1$.
Therefore for $\gamma^\vee=e_1-e_i\in\Delta^{*\vee}$ ($i\neq j$)
and $\gamma^\vee=e_1+e_i\in\Delta^{*\vee}$,
\begin{equation}
  \begin{split}
    p^*_{\mathbf{V}_I^\perp}(\gamma^\vee)
    &=e_1\pm e_i-\langle e_1\pm e_i,\lambda_1\rangle(e_1-e_j)
    \\
    &=e_1\pm e_i-(e_1-e_j)
    \\
    &=e_j\pm e_i\in\Delta_I^\vee,
  \end{split}
\end{equation}
and 
\begin{equation}
  \begin{split}
    p^*_{\mathbf{V}_I^\perp}(2e_1)
    &=2e_1-\langle 2e_1,\lambda_1\rangle(e_1-e_j)
    \\
    &=2e_1-2(e_1-e_j)
    \\
    &=2e_j\in\Delta_I^\vee.
  \end{split}
\end{equation}
Next for $\beta^\vee=e_1+e_j$ ($2\leq j\leq r$), we see that
\begin{equation}
  b_l=\langle\beta^\vee,\lambda_l\rangle=
  \begin{cases}
    1\qquad&(1\leq l<j)\\
    2\qquad&(j\leq l<r)\\
    1\qquad&(l=r),
  \end{cases}
\end{equation}
and in particular $\mu_\beta^{\mathbf{V}}=\lambda_1=e_1$.
Therefore for $\gamma^\vee=e_1+e_i\in\Delta^{*\vee}$ ($i\neq j$)
and $\gamma^\vee=e_1-e_i\in\Delta^{*\vee}$,
\begin{equation}
  \begin{split}
    p^*_{\mathbf{V}_I^\perp}(\gamma^\vee)
    &=e_1\pm e_i-\langle e_1\pm e_i,\lambda_1\rangle(e_1+e_j)
    \\
    &=e_1\pm e_i-(e_1+e_j)
    \\
    &=-e_j\pm e_i\in\Delta_I^\vee,
  \end{split}
\end{equation}
and 
\begin{equation}
  \begin{split}
    p^*_{\mathbf{V}_I^\perp}(2e_1)
    &=2e_1-\langle 2e_1,\lambda_1\rangle(e_1+e_j)
    \\
    &=2e_1-2(e_1+e_j)
    \\
    &=-2e_j\in\Delta_I^\vee.
  \end{split}
\end{equation}
Finally for $\beta^\vee=2e_1$, we see that
\begin{equation}
  b_l=\langle\beta^\vee,\lambda_l\rangle=
  \begin{cases}
    2\qquad&(1\leq l<r)\\
    1\qquad&(l=r),
  \end{cases}
\end{equation}
and in particular $\mu_\beta^{\mathbf{V}}=\lambda_1/2=e_1/2$.
Therefore for $\gamma^\vee=e_1\pm e_i\in\Delta^{*\vee}$,
\begin{equation}
  \begin{split}
    p^*_{\mathbf{V}_I^\perp}(\gamma^\vee)
    &=e_1\pm e_i-\langle e_1\pm e_i,\lambda_1/2\rangle 2e_1
    \\
    &=e_1\pm e_i-e_1
    \\
    &=\pm e_i\in\Delta_I^\vee/2.
  \end{split}
\end{equation}

Using the above data of $p^*_{\mathbf{V}_I^\perp}(\gamma^\vee)$ and 
\eqref{lambda_B}, we can calculate
$$
\langle p^*_{\mathbf{V}_I^\perp}(\gamma^\vee), \lambda\rangle
=\langle p^*_{\mathbf{V}_I^\perp}(\gamma^\vee), m_2\lambda_2+\cdots+
m_r\lambda_r\rangle.
$$
Let
\begin{align*}
m_{j,k}=\left\{
  \begin{array}{ll}
  m_j+\cdots+m_k  &  (j\leq k), \\
  0  &  ({\rm otherwise}).
  \end{array}\right.
\end{align*}
(Notice that this notation is different from $m_{j,k}$ used in \cite[(32)]{KMT-Poincare}.)
Then we obtain that, when $\beta^{\vee}=e_1-e_j$ ($2\leq j\leq r$), then
\begin{align*}
\langle p^*_{\mathbf{V}_I^\perp}(\gamma^\vee), \lambda\rangle=\left\{
  \begin{array}{ll}
    -m_{i,j-1} & (\gamma^{\vee}=e_1-e_i, i<j)\\
    m_{j,i-1} & (\gamma^{\vee}=e_1-e_i, j<i)\\
    m_{i,j-1}+2m_{j,r-1}+m_r & (\gamma^{\vee}=e_1+e_i,i\leq j)\\
    m_{j,i-1}+2m_{i,r-1}+m_r & (\gamma^{\vee}=e_1+e_i,j<i)\\
    2m_{j,r-1}+m_r & (\gamma^{\vee}=2e_1),
  \end{array}\right.
  \end{align*} 
when $\beta^{\vee}=e_1+e_j$ ($2\leq j\leq r$), then
\begin{align*}
\langle p^*_{\mathbf{V}_I^\perp}(\gamma^\vee), \lambda\rangle=\left\{
  \begin{array}{ll}
    -(m_{i,j-1}+2m_{j,r-1}+m_r) & (\gamma^{\vee}=e_1-e_i,i\leq j)\\
    -(m_{j,i-1}+2m_{i,r-1}+m_r) & (\gamma^{\vee}=e_1-e_i,j<i)\\
    m_{i,j-1} & (\gamma^{\vee}=e_1+e_i, i<j)\\
    -m_{j,i-1} & (\gamma^{\vee}=e_1+e_i, j<i)\\
    -(2m_{j,r-1}+m_r) & (\gamma^{\vee}=2e_1),
  \end{array}\right.
  \end{align*} 
and when $\beta^{\vee}=2e_1$, then
$$
\langle p^*_{\mathbf{V}_I^\perp}(\gamma^\vee), \lambda\rangle=
\pm\left(m_{i,r-1}+\frac{1}{2}m_r\right)
$$
for $\gamma^{\vee}=e_1\pm e_i$.

Therefore by \eqref{Fq-general} we now obtain the explicit form of the generating
function in the $B_r$ case.
We write $t_{e_1\pm e_i}=t_{\pm i}$ for $2\leq i\leq r$ and $t_{e_1}=t_1$.
Then the explicit form is as follows.
\begin{align}
  \label{Br-gene-F}
   & F(t_1,(t_{\pm i})_{2\leq i\leq r},(y_j)_{1\leq j\leq r},(m_i)_{2\leq i\leq r};\{2,\ldots,r\};B_r)  
\\
  &=
\sum_{j=2}^{r}
\prod_{2\leq i<j}\frac{t_{-i}}{t_{-i}-t_{-j}+2\pi\sqrt{-1}m_{i,j-1}}
\prod_{j<i\leq r}\frac{t_{-i}}{t_{-i}-t_{-j}-2\pi\sqrt{-1}m_{j,i-1}}
\notag\\
&\qquad\times
\prod_{2\leq i\leq j}\frac{t_{+i}}{t_{+i}-t_{-j}-2\pi\sqrt{-1}(m_{i,j-1}+2m_{j,r-1}+m_r)}
\notag\\
&\qquad\times
\prod_{j<i\leq r}\frac{t_{+i}}{t_{+i}-t_{-j}-2\pi\sqrt{-1}(m_{j,i-1}+2m_{i,r-1}+m_r)}
\notag\\
&\qquad\times\frac{t_1}{t_1-2t_{-j}-2\pi\sqrt{-1}(2m_{j,r-1}+m_{r})}
\notag\\
&\qquad\times
\exp\Bigl(2\pi\sqrt{-1}
\Bigl(
    \sum_{i=2}^{j-1}m_i(y_i-y_1)+\sum_{i=j}^r m_iy_i
\Bigr)
\Bigr)
\frac{t_{-j}\exp(t_{-j}\{y_1\})}{e^{t_{-j}}-1}
\notag\\
&+
\sum_{j=2}^{r}
\prod_{2\leq i\leq j}\frac{t_{-i}}{t_{-i}-t_{+j}+2\pi\sqrt{-1}(m_{i,j-1}+2m_{j,r-1}+m_r)}
\notag\\
&\qquad\times
\prod_{j<i\leq r}\frac{t_{-i}}{t_{-i}-t_{+j}+2\pi\sqrt{-1}(m_{j,i-1}+2m_{i,r-1}+m_r)}
\notag\\
&\qquad\times
\prod_{2\leq i<j}\frac{t_{+i}}{t_{+i}-t_{+j}-2\pi\sqrt{-1}m_{i,j-1}}
\prod_{j<i\leq r}\frac{t_{+i}}{t_{+i}-t_{+j}+2\pi\sqrt{-1}m_{j,i-1}}
\notag\\
&\qquad\times\frac{t_1}{t_1-2t_{+j}+2\pi\sqrt{-1}(2m_{j,r-1}+m_{r})}
\notag\\
&\qquad\times
\exp\Bigl(2\pi\sqrt{-1}
\Bigl(
    \sum_{i=2}^{j-1}m_i(y_i-y_1)+\sum_{i=j}^{r-1} m_i(y_i-2y_1) +m_r(y_r-y_1)
\Bigr)
     \Bigr)
\notag\\
&\qquad\times     
\frac{t_{+j}\exp(t_{+j}\{y_1\})}{e^{t_{+j}}-1}
\notag\\
&+
\prod_{2\leq i\leq r}\frac{t_{-i}}{t_{-i}-\frac{1}{2}t_{1}+\pi\sqrt{-1}(2m_{i,r-1}+m_r)}
\notag\\
&\qquad\times
\prod_{2\leq i\leq r}\frac{t_{+i}}{t_{+i}-\frac{1}{2}t_{1}-\pi\sqrt{-1}(2m_{i,r-1}+m_r)}
\notag\\
&\qquad\times
\frac{1}{2}\Biggl(\exp\Bigl(2\pi\sqrt{-1}
\Bigl(
    \sum_{i=2}^{r-1}m_i(y_i-y_1)+m_r(y_r-\frac{1}{2}y_1)
\Bigr)
\Bigr)
\notag\\
&\qquad\qquad\qquad\times
\frac{t_{1}\exp(t_{1}\{\frac{1}{2}y_1\})}{e^{t_{1}}-1}
\notag\\
&\qquad\qquad+
\exp\Bigl(2\pi\sqrt{-1}
\Bigl(
     \sum_{i=2}^{r-1}m_i(y_i-y_1)+m_r(y_r-\frac{1}{2}(y_1+1))
\Bigr)
     \Bigr)
\notag\\
&\qquad\qquad\qquad\times
\frac{t_{1}\exp(t_{1}\{\frac{1}{2}(y_1+1)\})}{e^{t_{1}}-1}
           \Biggr).
           \notag
\end{align}

\subsection{$D_r$ Case}
We realize $\Delta_+^\vee=\{e_i\pm e_j~|~1\leq i<j\leq r\}$
and $\Psi^{\vee}=\{e_1-e_2,\ldots,e_{r-1}-e_r, e_{r-1}+e_r\}$. 
Then 
\begin{align}\label{lambda-D}
\lambda_i=\left\{
  \begin{array}{ll}
  e_1+\cdots+e_i  &  (1\leq i\leq r-2),\\
  \frac{1}{2}(e_1+\cdots+e_{r-1}-e_r)  &  (i=r-1),\\
  \frac{1}{2}(e_1+\cdots+e_{r-1}+e_r)  &  (i=r).
  \end{array}\right.
  \end{align}
Choose $I=\{2,\ldots,r\}$ and $I^c=\{1\}$ as in the following diagram.
\begin{equation}
  \label{diagram_D3}
\setlength{\unitlength}{1pt}\begin{picture}(0,0)
\put(149,2){\oval(238,28)}
\end{picture}
\VN{\alpha_1}\E\VN{\alpha_2}\E\V\EO\V\E\V\E\V
\EDV{\E}{\E}
\end{equation}
Then 
\begin{equation}
  \Psi_I^\vee=\{\alpha_2^\vee=e_2-e_3,\ldots,\alpha_{r-1}^\vee=e_{r-1}-e_r,\alpha_r^\vee=e_{r-1}+e_r\}
\end{equation}
and $\Delta^{*\vee}=\{e_1\pm e_j~|~2\leq j\leq r\}$. 
Hence 
\begin{equation}
  \mathscr{V}_I^\vee=\{\{e_1-e_j\}\cup\Psi_I^\vee\}_{2\leq j\leq r}
\cup\{\{e_1+e_j\}\cup\Psi_I^\vee\}_{2\leq j\leq r}.
\end{equation}
For $\beta^\vee=e_1-e_j$ ($2\leq j\leq r$), we see that
\begin{equation}
  b_l=\langle\beta^\vee,\lambda_l\rangle=
  \begin{cases}
    1\qquad&(l<j)\\
    0\qquad&(\text{otherwise}),
  \end{cases}
\end{equation}
and in particular $\mu_\beta^{\mathbf{V}}=\lambda_1=e_1$.
Therefore for $\gamma^\vee=e_1-e_i\in\Delta^{*\vee}$ ($i\neq j$)
and $\gamma^\vee=e_1+e_i\in\Delta^{*\vee}$,
\begin{equation}\label{atodechuui1}
  \begin{split}
    p^*_{\mathbf{V}_I^\perp}(\gamma^\vee)
    &=e_1\pm e_i-\langle e_1\pm e_i,\lambda_1\rangle(e_1-e_j)
    \\
    &=e_1\pm e_i-(e_1-e_j)
    \\
    &=
    \begin{cases}
      e_j\pm e_i\in\Delta_I^\vee\qquad&(i\neq j)\\
      2e_i\qquad&(i=j).
    \end{cases}
  \end{split}
\end{equation}
Next for $\beta^\vee=e_1+e_j$ ($2\leq j\leq r$), we see that
\begin{equation}
  b_l=\langle\beta^\vee,\lambda_l\rangle=
  \begin{cases}
    0\qquad&(l=r-1,j=r)\\
    2\qquad&(j\leq l\leq r-2)\\
    1\qquad&(\text{otherwise}),
  \end{cases}
\end{equation}
and in particular $\mu_\beta^{\mathbf{V}}=\lambda_1=e_1$.
Therefore for $\gamma^\vee=e_1+e_i\in\Delta^{*\vee}$ ($i\neq j$)
and $\gamma^\vee=e_1-e_i\in\Delta^{*\vee}$,
\begin{equation}\label{atodechuui2}
  \begin{split}
    p^*_{\mathbf{V}_I^\perp}(\gamma^\vee)
    &=e_1\pm e_i-\langle e_1\pm e_i,\lambda_1\rangle(e_1+e_j)
    \\
    &=e_1\pm e_i-(e_1+e_j)
    \\
    &=
    \begin{cases}
      -e_j\pm e_i\in\Delta_I^\vee\qquad&(i\neq j) \\
      -2e_j\qquad&(i=j).
    \end{cases}
  \end{split}
\end{equation}

Using the same notation as in the $B_r$ case,
we can now calculate
$\langle p^*_{\mathbf{V}_I^\perp}(\gamma^\vee), \lambda\rangle$
and
write down the explicit form of the generating function as follows.
\begin{align}
  \label{Dr-gene-F}
  &  F((t_{\pm i})_{2\leq i\leq r},(y_j)_{1\leq j\leq r},(m_i)_{2\leq i\leq r};\{2,\ldots,r\};D_r)  
\\
  &=
\sum_{j=2}^{r-1}
\prod_{2\leq i<j}\frac{t_{-i}}{t_{-i}-t_{-j}+2\pi\sqrt{-1}m_{i,j-1}}
    \prod_{j<i\leq r}\frac{t_{-i}}{t_{-i}-t_{-j}-2\pi\sqrt{-1}m_{j,i-1}}
    \notag\\
&\qquad\times
\prod_{2\leq i\leq j}\frac{t_{+i}}{t_{+i}-t_{-j}-2\pi\sqrt{-1}(m_{i,j-1}+2m_{j,r-2}+m_{r-1,r})}
\notag\\
&\qquad\times
\prod_{j<i<r}\frac{t_{+i}}{t_{+i}-t_{-j}-2\pi\sqrt{-1}(m_{j,i-1}+2m_{i,r-2}+m_{r-1,r})}
\notag\\
&\qquad\times
     \frac{t_{+r}}{t_{+r}-t_{-j}-2\pi\sqrt{-1}(m_{j,r-2}+m_{r})}
\notag\\
&\qquad\times
\exp\Bigl(2\pi\sqrt{-1}
\Bigl(
    \sum_{i=2}^{j-1}m_i(y_i-y_1)+\sum_{i=j}^r m_iy_i
\Bigr)
\Bigr)
\frac{t_{-j}\exp(t_{-j}\{y_1\})}{e^{t_{-j}}-1}
\notag  \\
  &+ 
\prod_{2\leq i<r}\frac{t_{-i}}{t_{-i}-t_{-r}+2\pi\sqrt{-1}m_{i,r-1}}
\notag\\
&\qquad\times
     \prod_{2\leq i<r}\frac{t_{+i}}{t_{+i}-t_{-r}-2\pi\sqrt{-1}(m_{i,r-2}+m_{r})}
\notag\\
&\qquad\times
     \frac{t_{+r}}{t_{+r}-t_{-r}-2\pi\sqrt{-1}(m_{r}-m_{r-1})}
\notag\\
&\qquad\times
\exp\Bigl(2\pi\sqrt{-1}
\Bigl(
    \sum_{i=2}^{r-1}m_i(y_i-y_1)+m_r y_r
\Bigr)
\Bigr)
\frac{t_{-r}\exp(t_{-r}\{y_1\})}{e^{t_{-r}}-1}
\notag\\
&+
\sum_{j=2}^{r-1}
\prod_{2\leq i\leq j}\frac{t_{-i}}{t_{-i}-t_{+j}+2\pi\sqrt{-1}(m_{i,j-1}+2m_{j,r-2}+m_{r-1,r})}
\notag\\
&\qquad\times
\prod_{j<i<r}\frac{t_{-i}}{t_{-i}-t_{+j}+2\pi\sqrt{-1}(m_{j,i-1}+2m_{i,r-2}+m_{r-1,r})}
\notag\\
&\qquad\times
\frac{t_{-r}}{t_{-r}-t_{+j}+2\pi\sqrt{-1}(m_{j,r-2}+m_{r})}
\notag\\
&\qquad\times
\prod_{2\leq i<j}\frac{t_{+i}}{t_{+i}-t_{+j}-2\pi\sqrt{-1}m_{i,j-1}}
\prod_{j<i\leq r}\frac{t_{+i}}{t_{+i}-t_{+j}+2\pi\sqrt{-1}m_{j,i-1}}
\notag\\
&\qquad\times
\exp\Bigl(2\pi\sqrt{-1}
\Bigl(
     \sum_{i=2}^{j-1} m_i(y_i-y_1) +
     \sum_{i=j}^{r-2}m_i(y_i-2y_1)
\notag  \\
  &\qquad\qquad+
     m_{r-1}(y_{r-1}-y_1)+
     m_{r}(y_{r}-y_1)
     \Bigr)
     \Bigr)
\notag  \\
& \qquad\times 
       \frac{t_{+j}\exp(t_{+j}\{y_1\})}{e^{t_{+j}}-1}
\notag  \\
  &+
    \prod_{2\leq i<r}\frac{t_{-i}}{t_{-i}-t_{+r}+2\pi\sqrt{-1}(m_{i,r-2}+m_{r})}
\notag  \\
  & \qquad\times 
    \frac{t_{-r}}{t_{-r}-t_{+r}-2\pi\sqrt{-1}(m_{r-1}-m_{r})}
\notag\\
&\qquad\times
\prod_{2\leq i<r}\frac{t_{+i}}{t_{+i}-t_{+r}-2\pi\sqrt{-1}m_{i,r-1}}
\notag\\
&\qquad\times
\exp\Bigl(2\pi\sqrt{-1}
\Bigl(
     \sum_{i=2}^{r-2} m_i(y_i-y_1) +m_{r-1} y_{r-1} +m_r(y_r-y_1) 
     \Bigr)
     \Bigr)
\notag  \\
& \qquad\times 
\frac{t_{+r}\exp(t_{+r}\{y_1\})}{e^{t_{+r}}-1}.
\end{align}

\begin{remark}

It should be noted that $\pm 2e_i$ 
appearing on \eqref{atodechuui1} and \eqref{atodechuui2} are not proportional to any coroots.
However they can be regarded as coroots in the root system of type $B_r$.
\end{remark}

\ 

\section{Generating Functions ($A_1^2\subset A_3$ Case)}\label{sec-4}

In this section, as another example, we consider the case of the root system of type
$A_3$ with 
$I=\{1,3\}$ and $I^c=\{2\}$, as in the following diagram.
\begin{equation}
  \label{diagram_A3}
\setlength{\unitlength}{1pt}\begin{picture}(0,0)\put(2,2){\circle{26}}%
\put(86,2){\circle{26}}%
\end{picture}
  \VN{\alpha_1}\E\VN{\alpha_2}\E\VN{\alpha_3}
\end{equation}
Then 
\begin{equation}
    \Psi_I^\vee=\{\alpha_1^\vee,\alpha_3^\vee\}
\end{equation}
and 
$$\Delta^{*\vee}=\{\alpha_2^\vee,\alpha_1^\vee+\alpha_2^\vee,\alpha_2^\vee+\alpha_3^\vee,\alpha_1^\vee+\alpha_2^\vee+\alpha_3^\vee\}=\{e_i-e_j\}_{1\leq i\leq 2<j\leq4}.$$
Hence 
\begin{equation}
  \mathscr{V}_I^\vee=\{\{\beta^\vee\}\cup\Psi_I^\vee\}_{\beta^\vee\in\Delta^{*\vee}}
\end{equation}
We see that
\begin{gather}
  b_l(\alpha_1+\alpha_2=e_1-e_3)=\langle\alpha_1^\vee+\alpha_2^\vee,\lambda_l\rangle=
  \begin{cases}
    1\qquad&(l=1,2)\\
    0\qquad&(l=3),
  \end{cases}
  \\
  b_l(\alpha_2=e_2-e_3)=\langle\alpha_2^\vee,\lambda_l\rangle=
  \begin{cases}
    1\qquad&(l=2)\\
    0\qquad&(l=1,3),
  \end{cases}
  \\
  b_l(\alpha_2+\alpha_3=e_2-e_4)=\langle\alpha_2^\vee+\alpha_3^\vee,\lambda_l\rangle=
  \begin{cases}
    1\qquad&(l=2,3)\\
    0\qquad&(l=1),
  \end{cases}
  \\
  b_l(\alpha_1+\alpha_2+\alpha_3=e_1-e_4)=\langle\alpha_1^\vee+\alpha_2^\vee+\alpha_3^\vee,\lambda_l\rangle=1,
\end{gather}
and in particular $\mu_\beta^{\mathbf{V}}=\lambda_2$.
Therefore for $\gamma^\vee\in\Delta^{*\vee}\setminus\{\beta^\vee\}$,
\begin{equation}
  \begin{split}
    p^*_{\mathbf{V}_I^\perp}(\gamma^\vee)
    &=\gamma^\vee-\langle \gamma^\vee,\lambda_2\rangle\beta^\vee
    \\
    &=\gamma^\vee-\beta^\vee.
  \end{split}
\end{equation}
By putting $t_{e_i-e_j}=t_{ij}$ for $1\leq i\leq 2<j\leq 4$,
we obtain the generating function as follows.
\begin{align}\label{GF:A3}
  &  F((t_{13},t_{23},t_{23},t_{14}),(y_1,y_2,y_3),(m_1,m_3);\{1,3\};A_3)  
\\
  &=
  \frac{t_{23}}{t_{23}-t_{13}+2\pi\sqrt{-1}m_1}
  \frac{t_{14}}{t_{14}-t_{13}-2\pi\sqrt{-1}m_3}\notag\\
&\qquad\times
  \frac{t_{24}}{t_{24}-t_{13}-2\pi\sqrt{-1}(m_3-m_1)}
  \notag\\
&\qquad\times
\exp\Bigl(2\pi\sqrt{-1}
\Bigl(
    m_1(y_1-y_2)+m_3y_3
\Bigr)
\Bigr)
\frac{t_{13}\exp(t_{13}\{y_2\})}{e^{t_{13}}-1}
\notag\\
  &+
  \frac{t_{13}}{t_{13}-t_{23}-2\pi\sqrt{-1}m_1}
  \frac{t_{14}}{t_{14}-t_{23}-2\pi\sqrt{-1}(m_1+m_3)}\notag\\
&\qquad\times
  \frac{t_{24}}{t_{24}-t_{23}-2\pi\sqrt{-1}m_3}
  \notag\\
&\qquad\times
\exp\Bigl(2\pi\sqrt{-1}
\Bigl(
    m_1y_1+m_3y_3
\Bigr)
\Bigr)
\frac{t_{23}\exp(t_{23}\{y_2\})}{e^{t_{23}}-1}
\notag\\
  &+
  \frac{t_{13}}{t_{13}-t_{24}-2\pi\sqrt{-1}(m_1-m_3)}
  \frac{t_{14}}{t_{14}-t_{24}-2\pi\sqrt{-1}m_1}\notag\\
&\qquad\times
  \frac{t_{23}}{t_{23}-t_{24}+2\pi\sqrt{-1}m_3}
  \notag\\
&\qquad\times
\exp\Bigl(2\pi\sqrt{-1}
\Bigl(
    m_1y_1+m_3(y_3-y_2)
\Bigr)
\Bigr)
\frac{t_{24}\exp(t_{24}\{y_2\})}{e^{t_{24}}-1}
\notag\\
  &+
  \frac{t_{13}}{t_{13}-t_{14}+2\pi\sqrt{-1}m_3}
  \frac{t_{23}}{t_{23}-t_{14}+2\pi\sqrt{-1}(m_1+m_3)}\notag\\
&\qquad\times
  \frac{t_{24}}{t_{24}-t_{14}+2\pi\sqrt{-1}m_1}
  \notag\\
&\qquad\times
\exp\Bigl(2\pi\sqrt{-1}
\Bigl(
    m_1(y_1-y_2)+m_3(y_3-y_2)
\Bigr)
\Bigr)
\frac{t_{14}\exp(t_{14}\{y_2\})}{e^{t_{14}}-1}.\notag
\end{align}

\ 

\section{Several examples of functional relations}\label{sec-5}

From the explicit forms of generating functions given in the preceding two sections,
we can deduce various functional relations.    From the results proved in Section
\ref{sec-3}, we can show the general forms of functional relations in the cases of
types $B_r$ and $D_r$, similar to \cite[Theorem 3.2]{KMT-Poincare}, though we do not give the statement in the present paper. 
We here give explicit examples of generating functions of type $B_3$ and also of type $A_3$ $(\simeq D_3)$.

\subsection{$B_3$ Case}

We write zeta-functions of root systems of type $B_2$ and type $B_3$ as
\begin{align*}
& \zeta_2(s_1,s_2,s_3,s_4;B_2) =\sum_{m_1,m_2\geq 1}\frac{1}{m_1^{s_1}m_2^{s_2}(m_1+m_2)^{s_3}(2m_1+m_2)^{s_4}},\\
  &\zeta_3(s_1,s_2,s_3,s_4,s_5,s_6,s_7,s_8,s_9;B_3)\\
  &\qquad=\sum_{m_1,m_2,m_3\geq 1}\frac{1}{m_1^{s_1}m_2^{s_2}m_3^{s_3}(m_1+m_2)^{s_4}(m_2+m_3)^{s_5}(2m_2+m_3)^{s_6}}\\
  & \qquad \times \frac{1}{(m_1+m_2+m_3)^{s_7}(m_1+2m_2+m_3)^{s_8}(2m_1+2m_2+m_3)^{s_9}}.
\end{align*}
These series converge absolutely for $\Re s_1,\ldots,\Re s_9\geq 1$ because
$m_1+m_2\geq2\sqrt{m_1m_2}$ and $m_1+m_2+m_3\geq3\sqrt[3]{m_1m_2m_3}$.

Here we give explicit forms of functional relations among them as follows.
Let $r=3$, $\Delta=\Delta(B_3)$, $I=\{ 2,3\}$ and $(y_1,y_2,y_3)=(0,0,0)$ 
in \eqref{Br-gene-F}. Then we have
  \begin{align}\label{B3-gene-F}
  &F((t_1,t_{\pm 2},t_{\pm 3}),\mathbf{0},(m_2,m_3);\{2,3\};B_3)  
  \\
  &= 
\frac{t_{-3}}{t_{-3}-t_{-2}-2\pi\sqrt{-1}m_2}
\frac{t_{+2}}{t_{+2}-t_{-2}-2\pi\sqrt{-1}(2m_2+m_3)}
\notag\\
&\qquad\times
\frac{t_{+3}}{t_{+3}-t_{-2}-2\pi\sqrt{-1}(m_2+m_3)}\notag\\
&\qquad\times \frac{t_1}{t_1-2t_{-2}-2\pi\sqrt{-1}(2m_2+m_{3})}
\frac{t_{-2}}{e^{t_{-2}}-1}
\notag\\
&+ 
\frac{t_{-2}}{t_{-2}-t_{-3}+2\pi\sqrt{-1}m_2}
\frac{t_{+2}}{t_{+2}-t_{-3}-2\pi\sqrt{-1}(m_2+m_3)}
\notag\\
&\qquad\times
\frac{t_{+3}}{t_{+3}-t_{-3}-2\pi\sqrt{-1}m_3}
\frac{t_1}{t_1-2t_{-3}-2\pi\sqrt{-1}m_{3}}
\frac{t_{-3}}{e^{t_{-3}}-1}
\notag\\
&+ 
\frac{t_{-2}}{t_{-2}-t_{+2}+2\pi\sqrt{-1}(2m_2+m_3)}
\frac{t_{-3}}{t_{-3}-t_{+2}+2\pi\sqrt{-1}(m_2+m_3)}
\notag\\
&\qquad\times
\frac{t_{+3}}{t_{+3}-t_{+2}+2\pi\sqrt{-1}m_2}
\frac{t_1}{t_1-2t_{+2}+2\pi\sqrt{-1}(2m_2+m_{3})}
\frac{t_{+2}}{e^{t_{+2}}-1}
\notag\\
&+ 
\frac{t_{-2}}{t_{-2}-t_{+3}+2\pi\sqrt{-1}(m_2+m_3)}
\frac{t_{-3}}{t_{-3}-t_{+3}+2\pi\sqrt{-1}m_3}
\notag\\
&\qquad\times
\frac{t_{+2}}{t_{+2}-t_{+3}-2\pi\sqrt{-1}m_2}
\frac{t_1}{t_1-2t_{+3}+2\pi\sqrt{-1}m_{3}}
\frac{t_{+3}}{e^{t_{+3}}-1}
\notag\\
&+
\frac{t_{-2}}{t_{-2}-\frac{1}{2}t_{1}+\pi\sqrt{-1}(2m_2+m_3)}
\frac{t_{-3}}{t_{-3}-\frac{1}{2}t_{1}+\pi\sqrt{-1}m_3}
\notag\\
&\qquad\times
\frac{t_{+2}}{t_{+2}-\frac{1}{2}t_{1}-\pi\sqrt{-1}(2m_2+m_3)}
\frac{t_{+3}}{t_{+3}-\frac{1}{2}t_{1}-\pi\sqrt{-1}m_3}
\notag\\
&\qquad\times
\frac{1}{2}\Bigl(\frac{t_{1}}{e^{t_{1}}-1}
+
(-1)^{m_3}
\frac{t_{1}\exp(\frac{1}{2}t_{1})}{e^{t_{1}}-1}
\Bigr).\notag
\end{align}
Hence we can compute $P(\mathbf{k},\mathbf{y},\lambda;I;B_3)$.
For example, we obtain
\begin{align}
 & P((2,1,1,1,1),\mathbf{0},(m_2,m_3);\{2,3\};B_3)
  \\
    &=\frac{1}{16 \pi ^6 m_2^2 m_3^3 (m_2+m_3)}
    +\frac{1}{16 \pi ^6 m_2^2 (m_2+m_3) (2 m_2+m_3)^3}\notag\\
    &-\frac{5}{16 \pi ^6 m_2 m_3^4 (m_2+m_3)}
    \notag\\
    &-\frac{(-1)^{m_3}}{4 \pi ^6 m_3^4 (2 m_2+m_3)^2}
    -\frac{1}{4 \pi ^6 m_3^4 (2 m_2+m_3)^2}
    -\frac{1}{16 \pi ^6 m_2 m_3^3 (m_2+m_3)^2}
    \notag\\
    &-\frac{(-1)^{m_3}}{24 \pi ^4 m_3^2 (2 m_2+m_3)^2}
    -\frac{(-1)^{m_3}}{4 \pi ^6 m_3^2 (2 m_2+m_3)^4}
    -\frac{1}{4 \pi ^6 m_3^2 (2 m_2+m_3)^4}
    \notag\\
    &+\frac{1}{12 \pi ^4 m_3^2 (2 m_2+m_3)^2}
    +\frac{1}{16 \pi ^6 m_2 (m_2+m_3)^2 (2 m_2+m_3)^3}
    \notag\\
    &+\frac{5}{16 \pi ^6 m_2 (m_2+m_3) (2 m_2+m_3)^4}.\notag
  \end{align}
Therefore, from Theorem \ref{thm:main1} and by using the relation
\begin{align*}
&\sum_{m,n\geq 1}\frac{(-1)^n}{m^{s_1}n^{s_2}(m+n)^{s_3}(2m+n)^{s_4}}\\
&\ =-\sum_{m,n\geq1}\frac{1}{m^{s_1}n^{s_2}(m+n)^{s_3}(2m+n)^{s_4}}\\
&\quad +2\sum_{m,n\geq 1}\frac{1}{m^{s_1}(2n)^{s_2}(m+2n)^{s_3}(2m+2n)^{s_4}}\\
&\ =-\zeta_2(s_1,s_2,s_3,s_4;B_2)+2^{1-s_2-s_4}\zeta_2(s_2,s_1,s_4,s_3;B_2),
\end{align*}
we obtain the functional relation 
\begin{align}
  \label{B3-fr-1}
  & \ \ \ \zeta_3(1,s_2,s_3,1,s_5,s_6,1,1,2;B_3)-\zeta_3(1,1,s_3,s_2,1,2,s_5,1,s_6;B_3)\\
\notag
  & +\zeta_3(s_2,1,2,1,1,s_3,1,s_5,s_6;B_3)+\zeta_3(s_2,1,2,1,1,s_3,1,s_5,s_6;B_3)\\
  \notag
  & -\zeta_3(1,1,s_3,s_2,1,2,s_5,1,s_6;B_3)+\zeta_3(1,s_2,s_3,1,s_5,s_6,1,1,2;B_3)
    \\
\notag
  & =
    (-1)^5\frac{(2\pi\sqrt{-1})^{6}}{2!1!1!1!1!}
    \sum_{m_2=1}^\infty
    \sum_{m_3=1}^\infty \frac{P((2,1,1,1,1),\mathbf{0},(m_2,m_3);\{2,3\};B_3)}{m_2^{s_2} m_3^{s_3} (m_2+m_3)^{s_5} (2m_2+m_3)^{s_6}}\notag\\
  \notag
  &= 2\zeta_2(s_2+2,s_3+3,s_5+1,s_6;B_2)
+2\zeta_2(s_2+2,s_3,s_5+1,s_6+3;B_2)\notag\\
&-10\zeta_2(s_2+1,s_3+4,s_5+1,s_6;B_2)
-\frac{2^{-s_3-s_6}}{4}\zeta_2(s_3+4,s_2,s_6+2,s_5;B_2)\notag\\
&-2\zeta_2(s_2+1,s_3+3,s_5+2,s_6;B_2)
+4\pi^2\zeta_2(s_2,s_3+2,s_5,s_6+2;B_2)\notag\\
&-\frac{2^{-s_3-s_6}\pi^2}{6}\zeta_2(s_3+2,s_2,s_6+2,s_5;B_2) -\frac{2^{-s_3-s_6}}{4}\zeta_2(s_3+2,s_2,s_6+4,s_5;B_2)\notag\\
&+2\zeta_2(s_2+1,s_3,s_5+2,s_6+3;B_2)
+10\zeta_2(s_2+1,s_3,s_5+1,s_6+4;B_2).\notag
\end{align}
Setting $(s_2,s_5,s_6)=(1,1,2)$, we obtain
\begin{align}
  & 2\zeta_3(1,1,2,1,1,s_3,1,1,2;B_3)\label{B3-fr-2}\\
  \notag
  &= 2\zeta_2(3,s_3+3,2,2;B_2)
+2\zeta_2(3,s_3,2,5;B_2)\notag\\
&-10\zeta_2(2,s_3+4,2,2;B_2)
-\frac{2^{-s_3}}{16}\zeta_2(s_3+4,1,4,1;B_2)\notag\\
&-2\zeta_2(2,s_3+3,3,2;B_2)
+4\pi^2\zeta_2(1,s_3+2,1,4;B_2)\notag\\
&-\frac{2^{-s_3}\pi^2}{24}\zeta_2(s_3+2,1,4,1;B_2) -\frac{2^{-s_3}}{16}\zeta_2(s_3+2,1,6,1;B_2)\notag\\
&+2\zeta_2(2,s_3,3,5;B_2)
+10\zeta_2(2,s_3,2,6;B_2)\notag
\end{align}
which corrensponds to the result for the $C_3$ case 
(see \cite[(97)]{KMT-Poincare}).

Here we recall the known fact that 
\begin{equation}
\zeta_2(a,b,c,d;B_2)\in \mathbb{Q}\left[\pi^2,\{\zeta(2j+1)\}_{j\in \mathbb{N}}\right]\label{B2-Parity}
\end{equation}
for $a,b,c,d\in \mathbb{N}$ with $2 \nmid (a+b+c+d)$, 
which was given by the third-named author (see \cite{Tsu04}). Similar to \cite[(100)]{KMT-Poincare}, setting $s_3=2k-1$ $(k\in \mathbb{N})$ in \eqref{B3-fr-2}, we obtain from \eqref{B2-Parity} that 
\begin{equation}
\begin{split}
&\zeta_3(1,1,2,1,1,2k-1,1,1,2;B_3)\in \mathbb{Q}\left[\pi^2,\{\zeta(2j+1)\}_{j\in \mathbb{N}}\right]
\end{split}
\label{B3-Parity}
\end{equation}
for $k\in \mathbb{N}$. For example, we obtain
\begin{align*}
& \zeta_3(1,1,2,1,1,1,1,1,2;B_3) =\frac{9\pi^4}{320}\zeta(7) - \frac{1429\pi^2}{384}\zeta(9) + \frac{4355}{128}\zeta(11),\\
&\zeta_3(1,1,2,1,1,3,1,1,2;B_3)\\
& \qquad =-\frac{7\pi^4}{320}\zeta(9) + \frac{5143\pi^2}{1536}\zeta(11) - \frac{15833}{512}\zeta(13),\\
& \zeta_3(1,1,2,1,1,5,1,1,2;B_3)\\
&\qquad =\frac{23\pi^8}{2419200}\zeta(7) + \frac{11\pi^6}{20160}\zeta(9) - \frac{941\pi^4}{15360}\zeta(11)\\
& \qquad\quad + \frac{16121\pi^2}{2048}\zeta(13) - \frac{74079}{1024}\zeta(15).
\end{align*}

\subsection{$A_3$ Case and $D_3$ Case}

Here we deduce an explicit functional relation involving the zeta-function of $A_3$, from the
result proved in Section \ref{sec-4}.
From \eqref{eq:def_F} and \eqref{GF:A3}, we can compute $P(\mathbf{k},\mathbf{y},\lambda;{1,3};A_3)$. For example, we obtain $P((1,1,1,1),\mathbf{0},(m_1,m_3);\lambda;{1,3};A_3)$ as follows. When $m_1\neq m_3$, we have
\begin{equation}
  \begin{split}
    &P((1,1,1,1),\mathbf{0},(m_1,m_3);\{1,3\};A_3)\\
    &\quad =\frac{1}{8 \pi ^4 m_1^2 m_3 (m_1+m_3)}
    +\frac{1}{8 \pi ^4 m_1^2 m_3 (m_3-m_1)}
    \\
    &\qquad -\frac{1}{8 \pi ^4 m_1 m_3^2 (m_3-m_1)}
    +\frac{1}{8 \pi ^4 m_1 m_3^2 (m_1+m_3)}
    \\
    &\qquad -\frac{1}{8 \pi ^4 m_1 m_3 (m_3-m_1)^2}
    +\frac{1}{8 \pi ^4 m_1 m_3 (m_1+m_3)^2}.
  \end{split}
\end{equation}
When $m_1=m_3=m$, we have 
\begin{equation}
    P((1,1,1,1),\mathbf{0},(m,m);\{1,3\};A_3)=
    \frac{7}{32 \pi ^4 m^4}-\frac{1}{48 \pi ^2 m^2}.
\end{equation}
Renaming the variables in the case $r=3$ of \eqref{def_A_r}, we write
\begin{align*}
& \zeta_3(s_1,s_2,s_3,s_4,s_5,s_6;A_3)\\
& =\sum_{m_1,m_2,m_3=1}^\infty \frac{1}{m_1^{s_1}m_2^{s_2}m_3^{s_3}(m_1+m_2)^{s_4}(m_2+m_3)^{s_5}(m_1+m_2+m_3)^{s_6}}.
\end{align*}
This series converges absolutely for $\Re s_1,\ldots,\Re s_6\geq 1$ because
 $m_1+m_2+m_3\geq3\sqrt[3]{m_1m_2m_3}$.

Then, from Theorem \ref{thm:main1}, we obtain the functional relation
\begin{align*}
   &\ \zeta_3(s_1,1,s_3,1,1,1;A_3) - \zeta_3(1,1,1,s_1,s_3,1;A_3)\\
   &\ \ + \zeta_3(1,s_3,1,1,1,s_1;A_3) +\zeta_3(s_3,1,s_1,1,1,1;A_3)  \\
   &\ \  - \zeta_3(1,1,1,s_3,s_1,1;A_3) + \zeta_3(1,s_1,1,1,1,s_3;A_3)\\
   & =(2\pi \sqrt{-1})^4\\
   & \qquad \times \bigg(\sum_{m_1,m_3=1 \atop m_1\neq m_3}^\infty \frac{1}{m_1^{s_1}m_3^{s_3}}P((1,1,1,1),\mathbf{0},(m_1,m_3);\{1,3\};A_3)\\
& \qquad \qquad +\sum_{m=1}^\infty \frac{1}{m^{s_1+s_3}}P((1,1,1,1),\mathbf{0},(m,m);\{1,3\};A_3)\bigg)\\
& =2\bigg\{\zeta_2(s_1+2,s_3+1,1;A_2) + \zeta_2(s_1+2,1,s_3+1;A_2) \\
& \quad -\zeta_2(s_3+1,1,s_1+2;A_2) - \zeta_2(s_1+1,1,s_3+2;A_2) \\
& \quad +\zeta_2(s_3+2,1,s_1+1;A_2) + \zeta_2(s_1+1,s_3+2,1;A_2) \\
& \quad -\zeta_2(s_1+1,2,s_3+1;A_2) - \zeta_2(s_3+1,2,s_1+1;A_2) \\
& \quad +\zeta_2(s_1+1,s_3+1,2;A_2)\bigg\} +\zeta(s_1+s_3+4)\\
& \quad -\frac{\pi^2}{3}\zeta(s_1+s_3+2).
\end{align*}
If we set $(s_1,s_3)=(1,1)$, we see that 
$$\{ \zeta_2(k,l,m;A_2)\mid k,l,m\in \mathbb{N}\ \text{with}\ k+l+m=6\}$$
appear on the right-hand side. Using the partial fraction decomposition formula
\begin{align*}
\frac{1}{X^pY^q}=& \sum_{i=0}^{p-1}\binom{q-1+i}{i}\frac{1}{X^{p-i}(X+Y)^{q+i}}\\
& +\sum_{i=0}^{q-1}\binom{p-1+i}{i}\frac{1}{Y^{q-i}(X+Y)^{p+i}}\quad (p,q\in \mathbb{N}),
\end{align*}
we obtain
\begin{align*}
\zeta_2(k,l,m;A_2)=& \sum_{i=0}^{k-1}\binom{l-1+i}{i}\zeta_{EZ,2}(k-i,l+m+i)\\
& +\sum_{i=0}^{l-1}\binom{k-1+i}{i}\zeta_{EZ,2}(l-i,k+m+i)
\end{align*}
(see Huard et al \cite[(1.6)]{HWZ}). 
It is well-known that $\zeta_{EZ,2}(p,q)$ $(p,q\in \mathbb{N},\ q\geq 2,\ p+q\leq 6)$ can be expressed in terms of $\{\zeta(m) \mid 2\leq m\leq 6\}$, by using EZ-Face, 
an interface for evaluation of Euler sums (see \cite{EZ-face}). Using these results, we can consequently obtain
\begin{align}
  2\,\zeta_3(1,1,1,1,1,1;A_3)&=4\zeta(3)^2-\frac{31}{5670}\pi^6.\label{A3-value}
\end{align}
We emphasize that this formula is a new result. 
It cannot be deduced from the known functional relation for type $A_3$ \cite[Theorem 9]{KMTPalanga}.

Here we give a comment on the $D_3$ case. Although generally we consider $r\geq 4$ for $D_r$ cases, the generating function \eqref{Dr-gene-F} is valid even when $r=3$.
In this case,
we see that the Dynkin diagram \eqref{diagram_D3} reduces to
\begin{equation}
\setlength{\unitlength}{1pt}\begin{picture}(0,0)
\put(48,2){\oval(30,40)}
\end{picture}
\VN{\alpha_1}\EDVN{\E}{\E}{\alpha_{2}}{\alpha_{3}}
\end{equation}
and this coincides with \eqref{diagram_A3}.
In fact, if we read $t_{13},t_{14},t_{23},t_{24}$, $y_1,y_2,y_3$, $m_1,m_3$ as $t_{-3},t_{+2},t_{-2},t_{+3}$, $y_2,y_1,y_3$, $m_2,m_3$ respectively in \eqref{GF:A3}, then we obtain the generating function for $D_3$ which comes from \eqref{Dr-gene-F}.

\section{Another type of functional relation for $\zeta_2({\bf s};C_2)$}\label{sec-6}

In this section, 
we give another type of functional relation for the 
zeta-function of type $C_2$ defined by 
\begin{equation}
\zeta_2(s_1,s_2,s_3,s_4;C_2)=\sum_{m,n=1}^\infty \frac{1}{m^{s_1}n^{s_2}(m+n)^{s_3}(m+2n)^{s_4}}, \label{def_C_2}
\end{equation}
which absolutely converges in the region $\Re s_j\geq 1$ ($j=1,3,4$) and $\Re s_2\geq 0$.

We first summarize the progress of research on this function. 

As stated in Section \ref{sec-1}, the second author defined \eqref{def_C_2} inspired by Zagier's observation on Witten's zeta-function in \cite{Zagier}, and showed its analytic continuation in \cite{MatBonn}. Based on this result, the third author evaluated \eqref{def_C_2} at any positive integer point $(s_1,s_2,s_3,s_4)=(k_1,k_2,k_3,k_4)$ where $k_1+k_2+k_3+k_4$ is odd (see \cite{Tsu04}). 
Furthermore, we gave general forms of functional relations for zeta-functions of root systems, and explicit examples for them including \eqref{def_C_2} (see \cite[Section 3]{KMTKyushu}, \cite[Section 3]{KMTPJA}). 
More explicit expressions for functional relations for $\zeta_2({\bf s};C_2)$ were given by Nakamura \cite{Nak08} and the authors \cite{KMTNicchuu}. As for their character analogues, see \cite{KMTLondon}. 

It is emphasized that these functional relations include Witten's volume formulas which imply $\zeta_2(2k,2k,2k,2k;C_2)\in \mathbb{Q}\cdot \pi^{8k}$ $(k\in \mathbb{N})$. On the other hand, these give no information on $\zeta_2(2k-1,2k-1,2k-1,2k-1;C_2)$ $(k\in \mathbb{N})$, because in these cases, the functional relations vanish. We here give a new type of functional relations between $\zeta_2({\bf s};C_2)$ and double zeta-functions of Euler--Zagier type, and evaluate $\zeta_2({\bf s};C_2)$ at positive integer points from the known result on double zeta values. This result especially gives an explicit expression formula for $\zeta_2(1,1,1,1;C_2)$ in terms of $\zeta(s)$ and polylogarithms (see \eqref{eq-6-4} below).

The result in this section is given by considering partial fraction decompositions and partial summations of harmonic sums, and the method here is totally different from that stated in the preceding sections. The main technique is similar to that used in \cite{Tsu04}
(and methodologically some common feature with \cite{ZBC,IM}, though
they studied the case of type $A_r$).

Let 
\begin{align*}
& \phi(s)=\sum_{m=1}^\infty \frac{(-1)^m}{m^s}=(2^{1-s}-1)\zeta(s),\\
& \zeta_{EZ,2}(s_1,s_2;\sigma_1,\sigma_2)=\sum_{1\leq m<n}\frac{\sigma_1^m~\sigma_2^n}{m^{s_1}n^{s_2}}\quad (\sigma_1,\sigma_2\in \{\pm 1\}).
\end{align*}
Note that $\zeta_{EZ,2}(s_1,s_2;1,1)=\zeta_{EZ,2}(s_1,s_2)$. We see that the function $\zeta_{EZ,2}(s_1,s_2;\sigma_1,\sigma_2)$ can be continued meromorphically
to the whole space $\mathbb{C}^2$.    In fact, this function is a special case of
twisted double zeta-functions, whose analytic properties have been studied by many
authors.     When $(\sigma_1,\sigma_2)=(-1,-1)$, the proof of the continuation is
included in the results of Akiyama and Ishikawa \cite{AkiIshi} or of de Crisenoy
\cite{dC}, and when $(\sigma_1,\sigma_2)=(\pm 1, \mp 1)$, it is included in
\cite{Kom} or \cite{KMTLondon}.

By considering partial fraction decompositions and the partial summation, we prove the following result. 

\ 

\begin{prop}\label{FR:C2}
For $a,b,c\in \mathbb{N}$ and $s\in \mathbb{C}$,
\begin{align*}
\zeta_2&(a,s,b,c;C_2) \\
=& \sum_{\ell=0}^{b-1}\binom{c-1+\ell}{\ell}(-1)^\ell \\
& \quad \times \bigg\{(-1)^a\sum_{j=0}^{b-\ell-2}\binom{a-1+j}{j}\zeta_{EZ,2}(s+a+c+\ell+j,b-\ell-j)\\
& \quad\   +\sum_{j=0}^{a-2}\binom{b-\ell-1+j}{j}(-1)^j\zeta(s+b+c+j)\zeta(a-j)\\
& \quad\  -(-1)^a\binom{a+b-\ell-2}{b-\ell-1}\\
& \quad\  \ \times \left\{\zeta_{EZ,2}(1,s+a+b+c-1)+\zeta(s+a+b+c)\right\}\bigg\}\\
+& (-1)^b\sum_{\ell=0}^{c-1}\binom{b-1+\ell}{\ell}2^{s+b+\ell-1} \\
& \quad \times \bigg\{(-1)^a\sum_{j=0}^{c-\ell-2}\binom{a-1+j}{j}\zeta_{EZ,2}(s+a+b+\ell+j,c-\ell-j)\\
& \quad\   +\sum_{j=0}^{a-2}\binom{c-\ell-1+j}{j}(-1)^j\zeta(s+b+c+j)\zeta(a-j)\\
& \quad\  -(-1)^a\binom{a+c-\ell-2}{c-\ell-1}\\
& \quad\  \ \times \left\{\zeta_{EZ,2}(1,s+a+b+c-1)+\zeta(s+a+b+c)\right\}\bigg\}\\
+& (-1)^b\sum_{\ell=0}^{c-1}\binom{b-1+\ell}{\ell}2^{s+b+\ell-1} \\
& \quad \times \bigg\{(-1)^a\sum_{j=0}^{c-\ell-2}\binom{a-1+j}{j}\\
& \qquad\qquad\times\zeta_{EZ,2}(s+a+b+\ell+j,c-\ell-j;-1,1)\\
& \quad\   +\sum_{j=0}^{a-2}\binom{c-\ell-1+j}{j}(-1)^j\phi(s+b+c+j)\zeta(a-j)\\
& \quad\  -(-1)^a\binom{a+c-\ell-2}{c-\ell-1}\\
& \quad\  \ \times \left\{\zeta_{EZ,2}(1,s+a+b+c-1;1,-1)+\phi(s+a+b+c)\right\}\bigg\}.
\end{align*}
\end{prop}

\begin{proof}
First we recall the partial fraction decomposition formula 
\begin{align}
\frac{1}{X^p(X+Y)^q}& =\sum_{i=0}^{p-1}\binom{q-1+i}{i}\frac{(-1)^i}{Y^{q+i}X^{p-i}} \label{PFD}\\
&\quad +(-1)^p\sum_{i=0}^{q-1}\binom{p-1+i}{i}\frac{1}{Y^{p+i}(X+Y)^{q-i}} \notag
\end{align}
for $p,q\in\mathbb{N}$
(see, for example, \cite[Lemma 1]{Borwein}). Furthermore, we consider the following 
``incomplete'' version of \eqref{PFD}, that is,
\begin{align}
\frac{1}{X^p(X+Y)^q}& =\sum_{i=0}^{p-2}\binom{q-1+i}{i}\frac{(-1)^i}{Y^{q+i}X^{p-i}} \label{PFD2}\\
& \quad +(-1)^{p-1}\binom{p+q-2}{p-1}\frac{1}{XY^{p+q-2}(X+Y)} \notag\\
&\quad +(-1)^p\sum_{i=0}^{q-2}\binom{p-1+i}{i}\frac{1}{Y^{p+i}(X+Y)^{q-i}}, \notag
\end{align}
where the empty sum is interpreted as $0$. 

For $a,b,c\in \mathbb{N}$ and $s\in \mathbb{C}$, using \eqref{PFD} with $(X,Y,p,q)=(m+n,n,b,c)$, we have
\begin{align}\label{prim_fq}
 \zeta_2& (a,s,b,c;C_2) =\sum_{i=0}^{b-1}\binom{c-1+i}{i}(-1)^i\sum_{m,n=1}^\infty  \frac{1}{m^a n^{s+c+i}(m+n)^{b-i}}\\
& \quad +(-1)^b \sum_{i=0}^{c-1}\binom{b-1+i}{i}\sum_{m,n=1}^\infty \frac{1}{m^a n^{s+b+i}(m+2n)^{c-i}},
\end{align}
where we denote the right-hand side by $I_{1}+I_{2}$. 
Applying \eqref{PFD2} with $(X,Y,p,q)=(m,n,a,b-i)$ to $I_1$, we have
\begin{align*}
I_{1}
&=\sum_{i=0}^{b-1}\binom{c-1+i}{i}(-1)^i\\
& \quad \times \bigg\{ (-1)^a\sum_{j=0}^{b-i-2}\binom{a-1+j}{j}\zeta_{EZ,2}(s+a+c+i+j,b-i-j)\\
& \qquad +(-1)^{a-1}\binom{a+b-i-2}{b-i-1}\sum_{m,n=1}^\infty \frac{1}{mn^{s+a+b+c-2}(m+n)}\\
& \qquad +\sum_{j=0}^{a-2}\binom{b-i-1+j}{j}(-1)^j\zeta(s+b+c+j)\zeta(a-j)\bigg\},
\end{align*}
and the second sum in the curly parentheses can be evaluated as
\begin{align}\label{choitokoitumo}
& \sum_{m,n=1}^\infty \frac{1}{mn^{s+a+b+c-2}(m+n)}\\
& \quad =\sum_{n=1}^\infty \frac{1}{n^{s+a+b+c-1}}\sum_{m=1}^\infty \left(\frac{1}{m}-\frac{1}{m+n}\right)\notag\\
& \quad =\sum_{\substack{m,n\\1\leq m\leq n}} \frac{1}{mn^{s+a+b+c-1}}\notag \\
& \quad = \zeta_{EZ,2}(1,s+a+b+c-1)+\zeta(s+a+b+c).\notag
\end{align}
As for $I_2$, corresponding to the decomposition
\begin{align*}
& \sum_{m,n=1}^\infty\frac{1}{m^a n^{s+b+i}(m+2n)^{c-i}}\\
& \quad =2^{s+b+i}\sum_{m,n=1}^\infty\frac{1}{m^a (2n)^{s+b+i}(m+2n)^{c-i}}\\
& \quad =2^{s+b+i-1}\sum_{m,n=1}^\infty\frac{1}{m^a n^{s+b+i}(m+n)^{c-i}}\\
& \qquad +2^{s+b+i-1}\sum_{m,n=1}^\infty\frac{(-1)^n}{m^a n^{s+b+i}(m+n)^{c-i}},
\end{align*}
we write $I_2=I_{21}+I_{22}$. Applying \eqref{PFD2} with $(X,Y,p,q)=(m,n,a,c-i)$ to $I_{21}$, we have
\begin{align*}
I_{21}&=(-1)^b\sum_{i=0}^{c-1}\binom{b-1+i}{i}2^{s+b+i-1}\sum_{m,n=1}^\infty \frac{1}{m^a n^{s+b+i}(m+n)^{c-i}}\\
&=(-1)^b\sum_{i=0}^{c-1}\binom{b-1+i}{i}2^{s+b+i-1}\\
& \quad \times \bigg\{ (-1)^a\sum_{j=0}^{c-i-2}\binom{a-1+j}{j}\zeta_{EZ,2}(s+a+b+i+j,c-i-j)\\
& \qquad +(-1)^{a-1}\binom{a+c-i-2}{c-i-1} \sum_{m,n=1}^\infty \frac{1}{mn^{s+a+b+c-2}(m+n)}\\
& \qquad +\sum_{j=0}^{a-2}\binom{c-i-1+j}{j}(-1)^j\zeta(s+b+c+j)\zeta(a-j)\bigg\}\\
&=(-1)^b\sum_{i=0}^{c-1}\binom{b-1+i}{i}2^{s+b+i-1}\\
& \quad \times \bigg\{ (-1)^a\sum_{j=0}^{c-i-2}\binom{a-1+j}{j}\zeta_{EZ,2}(s+a+b+i+j,c-i-j)\\
& \qquad +(-1)^{a-1}\binom{a+c-i-2}{c-i-1}\\
& \qquad \qquad\times\{\zeta_{EZ,2}(1,s+a+b+c-1)+\zeta(s+a+b+c)\}\\
& \qquad +\sum_{j=0}^{a-2}\binom{c-i-1+j}{j}(-1)^j\zeta(s+b+c+j)\zeta(a-j)\bigg\},
\end{align*}
where the last equality follows by using \eqref{choitokoitumo}.
Similarly we obtain
\begin{align*}
& I_{22}=(-1)^b\sum_{i=0}^{c-1}\binom{b-1+i}{i}2^{s+b+i-1}\\
& \qquad \times \bigg\{ (-1)^a\sum_{j=0}^{c-i-2}\binom{a-1+j}{j}\\
& \qquad\qquad\quad\times\zeta_{EZ,2}(s+a+b+i+j,c-i-j;-1,1)\\
& \qquad\qquad +(-1)^{a-1}\binom{a+c-i-2}{c-i-1}\zeta_{EZ,2}\\
& \qquad\qquad\quad  \times\{(1,s+a+b+c-1;1,-1)+\phi(s+a+b+c)\}\\
& \qquad\qquad +\sum_{j=0}^{a-2}\binom{c-i-1+j}{j}(-1)^j\phi(s+b+c+j)\zeta(a-j)\bigg\}.
\end{align*}
Combining these results, we consequently obtain the assertion.
\end{proof}

\begin{remark}
Partial fraction decompositions are very useful tool of finding functional relations.
For example, \eqref{prim_fq} is a very simple consequence of partial fraction decomposition,
but this formula is already a functional relation among the zeta-functions of
types $C_2$ and $A_2$.    (Note that \eqref{prim_fq} first appeared in \cite[(10)]{Tsu04}.)
Generally speaking, any double shuffle relations and associated functional equations
can be shown by using partial fraction decompositions (see \cite{KMTMathZ}).
\end{remark}
\ 

\begin{example}
Setting $(a,b,c)=(1,1,1)$ $(1,2,1)$ and $(3,3,3)$ in Proposition \ref{FR:C2}, we obtain
\begin{align}
& \zeta_2(1,s,1,1;C_2)=(1-2^s)\{\zeta_{EZ,2}(1,s+2)+\zeta(s+3)\}\label{eq-6-1}\\
& \qquad\qquad\qquad\qquad  -2^s\{\zeta_{EZ,2}(1,s+2;1,-1)+\phi(s+3)\},\notag\\
& \zeta_2(1,s,2,1;C_2)=2^{s+1}\{\zeta_{EZ,2}(1,s+3)+\zeta(s+4)\}\label{eq-6-2}\\
& \qquad\qquad\qquad\qquad +2^{s+1}\{\zeta_{EZ,2}(1,s+3;1,-1)+\phi(s+4)\} \notag\\
& \qquad\qquad\qquad\qquad -\zeta_{EZ,2}(s+2,2), \notag\\
& \zeta_2(3,s,3,3;C_2)\label{eq-6-3}\\
& \ =(19\cdot 2^{s+2}+8)\zeta_{EZ,2}(s+7,2) + (2^{s+2} - 1)\zeta_{EZ,2}(s+6,3) \notag\\
& \quad + 3(1-2^{s+6})\{\zeta_{EZ,2}(1,s+8)+\zeta(s+9)\}\notag\\
& \quad + (2^{s+2}\cdot 39+3)\zeta(2)\zeta(s+7)\notag\\
& \quad + (4-2^{s+2}\cdot 31)\zeta(3)\zeta(s+6) \notag\\
& \quad -192\{\zeta_{EZ,2}(1,s+8;1;-1)+\phi(s+9)\}+ 156\zeta(2)\phi(s+7)\notag\\
& \quad  - 124\zeta(3)\phi(s+6) \notag\\
& \quad + 2^{s+2}\{19 \zeta_{EZ,2}(s+7,2;-1,1) +\zeta_{EZ,2}(s+6,3;-1,1)\}. \notag
\end{align}
In particular, setting $s=1$ in \eqref{eq-6-1} and $s=0$ in \eqref{eq-6-2}, we have
\begin{align*}
&\zeta_2(1,1,1,1;C_2)=-2\zeta_{EZ,2}(1,3;1,-1)-\frac{1}{2}\zeta(2)^2+\frac{7}{4}\zeta(4),\\
&\zeta_2(1,0,2,1;C_2)=2\zeta_{EZ,2}(1,3;1,-1).
\end{align*}
Here we recall 
\begin{align*}
\zeta_{EZ,2}(1,3;1,-1)& =2{\rm Li}_4\left(\frac{1}{2}\right)+\frac{1}{12}(\log 2)^4-\frac{15}{8}\zeta(4)\\
& \quad +\frac{7}{4}\zeta(3)\log 2-\frac{1}{2}\zeta(2) (\log 2)^2,
\end{align*}
where ${\rm Li}_4(\cdot)$ denotes the polylogarithm of order $4$
(see \cite[Section 4]{Borwein}). Therefore we obtain
\begin{align}
\zeta_2(1,1,1,1;C_2) &=\frac{17}{10}\zeta(2)^2 - 4 {\rm Li}_4\left(\frac{1}{2}\right) - \frac{7}{2}\zeta(3) \log 2 \label{eq-6-4}\\
& \quad + \zeta(2) (\log 2)^2 -\frac{1}{6} (\log 2)^4, \notag\\
\zeta_2(1,0,2,1;C_2) &=-\frac{3}{2}\zeta(2)^2 + 4 {\rm Li}_4\left(\frac{1}{2}\right) + \frac{7}{2}\zeta(3) \log 2 \label{eq-6-4-2}\\
& \quad - \zeta(2) (\log 2)^2 +\frac{1}{6} (\log 2)^4. \notag
\end{align}
Similarly we can express, for example, 
$$\zeta_2(1,2,2,1;C_2)\quad \text{\rm and}\quad \zeta_2(3,3,3,3;C_2)$$
in terms of values of $\zeta(s)$ and $\zeta_{EZ,2}(s_1,s_2;\sigma_1,\sigma_2)$. However, it is unclear whether we can express these values in terms of single series like \eqref{eq-6-4} and \eqref{eq-6-4-2}. 
\end{example}

\


\begin{thebibliography}{999}
\bibitem{AkiIshi}
S. Akiyama and H. Ishikawa,
On analytic continuation of multiple $L$-functions and related zeta-functions,
in {\it Analytic Number Theory}, C. Jia and K. Matsumoto (eds.), Devel.\ Math.\ {\bf 6}, 
Kluwer Acad.\ Publ., 2002, 1--16.
\bibitem{Borwein}
D. Borwein, J. M. Borwein and R. Girgensohn, {Explicit evaluation of Euler sums}, 
{\it Proc.\ Edinburgh Math.\ Soc.} {\bf 38} (1995), 277--294.
\bibitem{EZ-face}
J. Borwein, P. Lisonek and P. Irvine, EZ-Face, \\
http://wayback.cecm.sfu.ca/projects/EZFace/credits.html.
\bibitem{dC}
M. de Crisenoy,
Values at $T$-tuples of negative integers of twisted multivariable zeta series
associated to polynomials of several variables,
{\it Compositio Math.} {\bf 142} (2006), 1371--1402.
\bibitem{Hoffman}
M. E. Hoffman,
Multiple harmonic series,
{\it Pacific J. Math.} {\bf 152} (1992), 275--290.
\bibitem{Hardy}
G. H. Hardy, Notes on some points in the integral calculus LV, On the integration of Fourier series,
{\it Messenger of Math.} {\bf 51} (1922), 186--192; reprinted in {\it Collected Papers of G. H. Hardy (including joint papers with J. E. Littlewood and others) Vol. III}, 
Clarendon Press, Oxford, 1969, pp.\ 506--512.
\bibitem{HWZ}
J. G. Huard, K. S. Williams and N.-Y. Zhang, On Tornheim's double series, 
{\it Acta Arith.} {\bf 75} (1996), 105--117. 
\bibitem{IM}
S. Ikeda and K. Matsuoka,
On functional relations for Witten multiple zeta-functions,
{\it Tokyo J. Math.} {\bf 39} (2016), 17--38.
\bibitem{Kom}
Y. Komori,
An integral representation of multiple Hurwitz--Lerch zeta functions and generalized 
multiple Bernoulli numbers,
{\it Quart.\ J. Math.\ (Oxford)} {\bf 61} (2010), 437--496.
\bibitem{KMTKyushu}
Y. Komori, K. Matsumoto and H. Tsumura, 
Zeta-functions of root systems,
in {\it The Conference on $L$-functions}, L. Weng and M. Kaneko (eds.), 
World Scientific, 2007, pp. 115--140. 
\bibitem{KMTPJA}
Y. Komori, K. Matsumoto and H. Tsumura, 
Zeta and $L$-functions and Bernoulli polynomials of root systems, 
{\it Proc.\ Japan Acad., Ser.\ A} {\bf 84} (2008), 62--67. 
\bibitem{KMTWitten2}
Y. Komori, K. Matsumoto and H. Tsumura, 
On Witten multiple zeta-functions associated with semisimple Lie algebras II, 
{\it J. Math.\ Soc.\ Japan} {\bf 62} (2010), 355--394.
\bibitem{KMTLondon}
Y. Komori, K. Matsumoto and H. Tsumura, 
On multiple Bernoulli polynomials and multiple $L$-functions of root systems,
{\it Proc.\ London Math.\ Soc.} {\bf 100} (2010), 303--347.
\bibitem{KMTNicchuu}
Y. Komori, K. Matsumoto and H. Tsumura, 
Functional relations for zeta-functions of root systems,
in {\it Number Theory: Dreaming in Dreams --- Proc. 5th China-Japan Seminar},
T. Aoki, S. Kanemitsu and J.-Y. Liu (eds.), Ser. on Number Theory and its Appl.
Vol. 6, World Scientific, 2010, pp. 135--183.
\bibitem{KMTDebrecen}
Y. Komori, K. Matsumoto and H. Tsumura,
Functional equations and functional relations for the Euler double zeta-function and its
generalization of Eisenstein type,
{\it Publ.\ Math.\ Debrecen} {\bf 77} (2010), 15--31.
\bibitem{KMTWitten4}
Y. Komori, K. Matsumoto and H. Tsumura, 
On Witten multiple zeta-functions associated with semisimple Lie algebras IV,
{\it Glasgow Math.\ J.} {\bf 53} (2011), 185--206.
\bibitem{KMTMathZ}
Y. Komori, K. Matsumoto and H. Tsumura, 
Shuffle products for multiple zeta values and partial fraction decompositions of
zeta-functions of root systems,
{\it Math.\ Z.} {\bf 268} (2011), 993--1011.
\bibitem{KMTIJNT}
Y. Komori, K. Matsumoto and H. Tsumura,
Functional equations for double $L$-functions and values at non-positive integers,
{\it Intern.\ J. Number Theory} {\bf 7} (2011), 1441--1461.
\bibitem{KMTWitten3}
Y. Komori, K. Matsumoto and H. Tsumura, 
On Witten multiple zeta-functions associated with semisimple Lie algebras III,
in {\it Multiple Dirichlet Series, $L$-functions and Automorphic Forms}, 
D. Bump et al.\ (eds.), Progr.\ in Math.\ Vol.\ 300, Springer, 2012, pp. 223--286.  
\bibitem{KMTPalanga}
Y. Komori, K. Matsumoto and H. Tsumura, 
Functional relations for zeta-functions of weight lattices of Lie groups of type $A\sb 3$, in {\it Analytic and Probabilistic Methods in Number Theory}, 
Proc.\ 5th Intern.\ Conf.\ in Honour of J. Kubilius, A. Laurin{\v c}ikas et al.\ (eds.),
TEV, Vilnius, 2012, pp.\ 151--172.
\bibitem{KMTFACM}
Y. Komori, K. Matsumoto and H. Tsumura, 
A study on multiple zeta values from the viewpoint of zeta-functions of root systems,
{\it Funct.\ Approx.\ Comment.\ Math.} {\bf 51} (2014), 43--76.
\bibitem{KMT2014}
Y. Komori, K. Matsumoto and H. Tsumura, 
Lattice sums of hyperplane arrangements, 
{\it Comment.\ Math.\ Univ.\ St.\ Pauli} \textbf{63} (2014), 161--213. 
\bibitem{KMTWitten5}
Y. Komori, K. Matsumoto and H. Tsumura, 
On Witten multiple zeta-functions associated with semisimple Lie algebras V,
{\it Glasgow Math. J.} {\bf 57} (2015), 107--130.
\bibitem{KMTKyiv}
Y. Komori, K. Matsumoto and H. Tsumura,
Zeta-functions of weight lattices of compact connected semisimple Lie groups,
in {\it Proc.\ 5th Intern.\ Conf.\ on Analytic Number Theory and Spatial Tesselations},
{\v S}iauliai Math.\ Semin.\ {\bf 10(18)} (2015), 149--179.
\bibitem{KMT-Poincare}
Y. Komori, K. Matsumoto and H. Tsumura, 
Zeta-functions of root systems and Poincar{\'e} polynomials of Weyl groups, to appear in Tohoku Math. J. (arXiv: 1707.09719).
\bibitem{MatBonn}
K. Matsumoto,
On Mordell--Tornheim and other multiple zeta-functions,
in {\it Proc. Session in Analytic Number Theory and Diophantine Equations},
D. R. Heath-Brown and B. Z. Moroz (eds.), Bonner Math.\ Schriften {\bf 360}, Univ.\ Bonn,
2003, n.\ 25, 17pp.
\bibitem{Mat04}
K. Matsumoto, 
Functional equations for double zeta-functions,
{\it Math. Proc. Cambridge Phil. Soc.} {\bf 136} (2004), 1--7.
\bibitem{Mat06}
K. Matsumoto,
Analytic properties of multiple zeta-functions in several variables,
in {\it Number Theory: Tradition and Modernization}, W. Zhang and Y. Tanigawa (eds.),
Springer, 2006, pp.\ 153--173.
\bibitem{MNOT}
K. Matsumoto, T. Nakamura, H. Ochiai and H. Tsumura,
On value-relations, functional relations and singularities of Mordell--Tornheim and
related triple zeta-functions,
{\it Acta Arith.} {\bf 132} (2008), 99--125.
\bibitem{MTFourier}
K. Matsumoto and H. Tsumura,
On Witten multiple zeta-functions associated with semisimple Lie algebras I,
{\it Ann.\ Inst.\ Fourier} {\bf 56} (2006), 1457--1504.
\bibitem{MTQJM}
K. Matsumoto and H. Tsumura,
A new method of producing functional relations among multiple zeta-functions, 
{\it Quart.\ J. Math.\ (Oxford)} {\bf 59} (2008), 55--83.
\bibitem{MTSiauliai}
K. Matsumoto and H. Tsumura,
Functional relations among certain double polylogarithms and their character analogues,
{\it {\v S}iauliai Math.\ Semin.} {\bf 3(11)} (2008), 189--205.
\bibitem{Mordell}
L. J. Mordell,
On the evaluation of some multiple series, 
{\it J. London Math.\ Soc.} {\bf 33} (1958), 368--371.
\bibitem{Nak06}
T. Nakamura,
A functional relation for the Tornheim double zeta function,
{\it Acta Arith.} {\bf 125} (2006), 257--263.
\bibitem{Nak08}
T. Nakamura,
Double Lerch value relations and functional relations for Witten zeta functions,
{\it Tokyo J. Math.} {\bf 31} (2008), 551--574.
\bibitem{Onod14}
K. Onodera,
A functional relation for Tornheim's double zeta functions,
{\it Acta Arith.} {\bf 162} (2014), 337--354.
\bibitem{Titch}
E. C. Titchmarsh, {\it The Theory of the Riemann Zeta-function}, 2nd ed. (revised by D. R. Heath-Brown), Oxford University Press, Oxford, 1986.
\bibitem{Tornheim}
L. Tornheim,
Harmonic double series,
{\it Amer. J. Math.} {\bf 72} (1950), 303--314.
\bibitem{Tsu04}
H. Tsumura.
On Witten's type of zeta values attached to $SO(5)$, 
{\it Arch.\ Math.\ (Basel)} {\bf 82} (2004), 147--152.
\bibitem{Tsu07}
H. Tsumura.
On functional relations between the Mordell--Tornheim double zeta-functions and the Riemann
zeta-function,
{\it Math.\ Proc.\ Cambridge Phil.\ Soc.} {\bf 142} (2007), 395--405.
\bibitem{Witten}
E. Witten,
On quantum gauge theories in two dimensions, 
{\it Commun. Math. Phys.} {\bf 141} (1991), 153--209.
\bibitem{Zagier}
D. Zagier,
Values of zeta functions and their applications,
in {\it First European Congress of Mathematics}, Vol. II, A. Joseph et al. (eds.),
Progr.\ in Math.\ Vol.\ 120, Birkh{\"a}user, 1994, pp.\ 497--512.
\bibitem{ZBC}
X. Zhou, D. M. Bradley and T. Cai, 
Depth reduction of a class of Witten zeta functions,
{\it Electron.\ J. Combin.} {\bf 16} (2009), no.\ 1, Note 27, 7pp. 
\end{thebibliography}
\end{document}